\documentclass[pdflatex,sn-mathphys-num]{sn-jnl}

\usepackage{amsfonts}
\usepackage{amsmath}
\usepackage{graphicx}
\usepackage{geometry}
\usepackage{csquotes}
\usepackage{amssymb}
\usepackage{algorithm}
\usepackage{algorithmicx}
\usepackage{algpseudocode}
\usepackage{tikz}
\usetikzlibrary{decorations.pathreplacing}
\usepackage{caption}

\algrenewcommand\algorithmicrequire{\textbf{Input:}}
\algrenewcommand\algorithmicensure{\textbf{Output:}}

\usepackage{color}
\usepackage{mathrsfs}
\newcommand{\R}{\mathbb{R}} 
\newcommand{\N}{\mathbb{N}} 
\newcommand{\de}{\mathrm{d}}

\renewcommand{\de}{\mathrm{d}} 

\newtheorem{theorem}{Theorem}[section]
\newtheorem{proposition}[theorem]{Proposition}
\newtheorem{lemma}[theorem]{Lemma}
\newtheorem{example}[theorem]{Example}
\newtheorem{remark}[theorem]{Remark}
\newtheorem{definition}[theorem]{Definition}
\newtheorem{corollary}[theorem]{Corollary}

\begin{document}

\title{On the conditioning of polynomial histopolation}

\author*[1]{\fnm{Ludovico} \sur{Bruni Bruno}}\email{ludovico.brunibruno@polito.it}

\author[2]{\fnm{Stefano} \sur{Serra-Capizzano}}\email{s.serracapizzano@uninsubria.it}

\affil*[1]{\orgdiv{Dipartimento di Scienze Matematiche \lq\lq Giuseppe Luigi Lagrange\rq\rq}, \orgname{Politecnico di Torino}, \orgaddress{\street{Corso Duca degli Abruzzi, 24}, \city{Torino}, \postcode{10129}, \country{Italia}}}

\affil[2]{\orgdiv{Dipartimento di Scienza e Alta Tecnologia}, \orgname{Universit\`a dell'Insubria}, \orgaddress{\street{Via Valleggio, 11}, \city{Como}, \postcode{22100}, \country{Italia}}}

\abstract{Histopolation is the approximation procedure that associates a degree $ d-1 $ polynomial $ p_{d-1} \in \mathscr{P}_{d-1} (I) $ with a locally integrable function $ f $ imposing that the integral (or, equivalently, the average) of $p$ coincides with that of $f$ on a collection of $ d $ distinct segments $s_i$.
	In this work we discuss unisolvence and conditioning of the associated matrices, in an asymptotic linear algebra perspective, i.e., when the matrix-size $d$ tends to infinity. While the unisolvence is a rather sparse topic, the conditioning in the unisolvent setting has a uniform behavior: as for the case of standard Vandermonde matrix-sequences with real nodes, the conditioning is inherently exponential as a function of $d$ when the monomial basis is chosen. In contrast, for an appropriate selection of supports, the Chebyshev polynomials of second kind exhibit a bounded conditioning. A linear behavior is also observed in the Frobenius norm.}



\maketitle

%



\section{Introduction}

Let $I \subset \R$ be a compact interval and let $ f : I \to \R $ be a locally integrable function. Histopolation is the approximation procedure that associates with $ f $ a degree $ d-1 $ polynomial $ p_{d-1} \in \mathscr{P}_{d-1} (I) $ such that
\begin{equation} \label{eq:histopolation}
	\int_{\alpha_i}^{\beta_i} f(x) \de x = \int_{\alpha_i}^{\beta_i} p(x) \de x ,
\end{equation}
for an appropriate collection of supports $ s_i := [\alpha_i, \beta_i] $, $ i = 1, \ldots, d $. In the literature \cite{BEfekete,Wendland}, an equivalent formulation of \eqref{eq:histopolation} containing normalization by the measure $ |s_i| $ of $ s_i $ is sometimes adopted; in what follows, we mainly stick to the non-normalized case. The main consequence of this choice is that the measures of the supports must be bounded away from zero; that is, we assume that there exists $ \varepsilon > 0 $ independent of $d$ such that $ |s_i| \geq \varepsilon $ for each $ i = 1, \ldots, d $. This is a rather natural assumption: even in the normalized case, it avoids the introduction of technical concepts such as \emph{regular} sets \cite{Robidoux}. However, we will also discuss situations where the normalization factor is present, with particular focus to limiting cases where segments collapse to points.

In contrast to the univariate nodal interpolation procedure, where any collection of $ d $ distinct points is unisolvent for polynomial interpolation in $ \mathscr P_{d-1} (I) $, such a neat recipe does not exist for histopolation, where the identification of specific classes is a hard task. It is indeed very easy to construct examples of $ d $ distinct segments that are not unisolvent for histopolation in $ \mathscr P_{d-1} (I) $. 

\begin{example}
	Consider the family of segments $[-i, i] $, for $ i = 1, \ldots, d $ with $ d \geq 1 $. Any non-vanishing linear combination of polynomials of odd degree less than or equal to $ d-1 $ satisfies
	$$ \int_{-i}^i q_{d-1} (x) \de x = 0 $$
	for $ i = 1, \ldots, d $. 
%
%
%
\end{example}

 This is only one of the many differences between nodal interpolation and histopolation. Such a dissimilarity is mostly motivated by the fact that any interval $ s_i $ requires \emph{two} parameters to be described (e.g., its endpoints $ \alpha_i $ and $ \beta_i $, or the midpoint $ \tau_i $ and the length $ |s_i| $), giving rise to a generally underdetermined problem. As a direct consequence of this, unisolvence results are confined to specific classes of segments. Many of them were identified in \cite{BEsinum}, and are:
\begin{itemize}
	\item[(C1)] chains of intervals. For $ i = 2, \ldots, d $, one has $ \alpha_i = \beta_{i-1} $; when $ \beta_i - \alpha_i = |s_i| $ is constant for $ i = 1, \ldots, d $, we call such a collection \emph{equidistributed};
		\begin{figure}[h]
		\centering	
		\begin{tikzpicture}[baseline=(current bounding box.north)]
			\coordinate (A) at (-3,0);
			\coordinate (B) at (3,0);
			\coordinate (C) at (1.5, 2.598);
			\coordinate (D) at (-1.5, 2.598);
			\coordinate (O) at (0,0);
			
			\draw[dashed] (A) -- (B);

			\coordinate (S2I) at (-2.298, 0);
			\coordinate (S2F) at (-0.521, 0);
			
			\coordinate (S1I) at (0.521, 0);
			\coordinate (S1F) at (2.298, 0);
			
			\coordinate (S3I) at (-0.521, 0);
			\coordinate (S3F) at (0.521, 0);

			\draw[thick, blue] (S1I) -- (S1F);
			\draw[thick, blue] (S2I) -- (S2F);
			\draw[thick, blue] (S3I) -- (S3F);
			
			\fill (S2I) circle[radius=1.5pt];
			\fill (S2F) circle[radius=1.5pt];
			\fill (S3F) circle[radius=1.5pt];
			\fill (S1F) circle[radius=1.5pt];
			
			\node at (-1.5,0) [anchor=south]{$s_1$};
			\node at (1.5,0) [anchor=south]{$s_3$};
			\node at (0,0) [anchor=south]{$s_2$};
			\node at (-3,0) [anchor=south east]{$a$};
			\node at (3,0) [anchor=south west]{$b$};
			
			\node at (S2I) [anchor = north]{\tiny $ \alpha_1 $};
			\node at (S2F) [anchor = north]{\tiny $ \beta_1 \equiv \alpha_2 $};
			\node at (S3F) [anchor = north]{\tiny $ \beta_2 \equiv \alpha_3 $};
			\node at (S1F) [anchor = north]{\tiny $ \beta_3 $};
		\end{tikzpicture}
		\caption{The segments $ s_1 $, $ s_2 $ and $ s_3 $ form a \emph{chain of intervals}. They are examples of non-equidistributed segments in the class (C1).} \label{fig:C1segments}
	\end{figure}
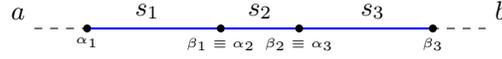
	\item[(C2)] segments with uniform arc-length. For some $ \tau_i \in (0, \pi) $ and $ 0 < \rho < \pi $, one has $ \alpha_i = \cos(\tau_i + \rho) $ and $ \beta_i = \cos(\tau_i - \rho) $. If the nodes $ \tau_i $ are the Chebyshev nodes and $ \rho = \pi/(2d)$, we refer to the corresponding collection as that of \emph{Chebyshev segments};
	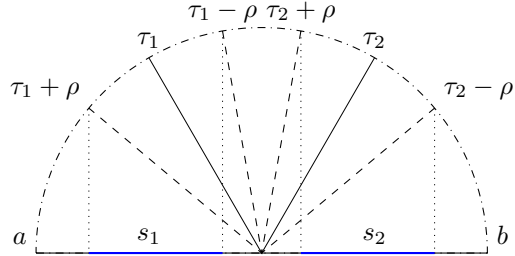
\begin{figure}[h]
		\centering	
		\begin{tikzpicture}[baseline=(current bounding box.north)]
			\draw[dashdotted] (-3,0) -- (3,0) arc(0:180:3) -- cycle;
			\coordinate (A) at (-3,0);
			\coordinate (B) at (3,0);
			\coordinate (C) at (1.5, 2.598);
			\coordinate (D) at (-1.5, 2.598);
			\coordinate (O) at (0,0);
			
			\draw (A) -- (B);
			
			\coordinate (C1) at (0.521, 2.954);
			\coordinate (C2) at (2.298, 1.928);
			
			\coordinate (D1) at (-0.521, 2.954);
			\coordinate (D2) at (-2.298, 1.928);
			\draw[ultra thin] (O) -- (C);
			\draw[ultra thin] (O) -- (D);
			
			\draw[dashed] (C1) -- (O) -- (C2);
			\draw[dashed] (D1) -- (O) -- (D2);
			
			\coordinate (S2I) at (-2.298, 0);
			\coordinate (S2F) at (-0.521, 0);
			
			\coordinate (S1I) at (0.521, 0);
			\coordinate (S1F) at (2.298, 0);
			
			\draw[dotted] (C1) -- (S1I);
			\draw[dotted] (C2) -- (S1F);
			\draw[dotted] (D2) -- (S2I);
			\draw[dotted] (D1) -- (S2F);
			
			\draw[thick, blue] (S1I) -- (S1F);
			\draw[thick, blue] (S2I) -- (S2F);
			
			\node at (-1.5,0) [anchor=south]{$s_1$};
			\node at (1.5,0) [anchor=south]{$s_2$};
			\node at (-3,0) [anchor=south east]{$a$};
			\node at (3,0) [anchor=south west]{$b$};
			\node at (1.5, 2.598) [anchor = south]{$\tau_2$};
			\node at (-1.5, 2.598) [anchor = south]{$\tau_1$};
			
			\node at (0.521, 2.954) [anchor = south]{$\tau_2 + \rho $};
			\node at (2.298, 1.928) [anchor = south west]{$\tau_2 - \rho$};
			
			\node at (-0.521, 2.954) [anchor = south]{$\tau_1 - \rho $};
			\node at (-2.298, 1.928) [anchor = south east]{$\tau_1 + \rho$};
		\end{tikzpicture}
		\caption{The segments $ s_1 $ and $ s_2 $ are the projections of arcs on the semicircle with constant arc-length $2 \rho$. They are examples of segments in the class (C2).} \label{fig:Chebsegments}
	\end{figure}
	\item[(C3)] segments with identical (left or right) endpoint. For $ i = 1, \ldots, d $, one has $ \alpha_i = \alpha $, for some $ \alpha \in I $, and $ \beta_i \ne \beta_j $ for $ i \ne j $. In the right endpoint case, for $ i = 1, \ldots, d $ one has $ \alpha_i \ne \alpha_j $ for $ i \ne j $ and $ \beta_i = \beta $  for some $ \beta \in I $.
			\begin{figure}[h]
		\centering	
		\begin{tikzpicture}[baseline=(current bounding box.north)]
			\coordinate (A) at (-3,0);
			\coordinate (B) at (3,0);
			\coordinate (C) at (1.5, 2.598);
			\coordinate (D) at (-1.5, 2.598);
			\coordinate (O) at (0,0);
			
			\draw[dashed] (A) -- (B);

			\coordinate (S2I) at (-2.298, 0);
			\coordinate (S2F) at (-0.521, 0);
			
			\coordinate (S1I) at (0.521, 0);
			\coordinate (S1F) at (2.298, 0);
			
			\coordinate (S3I) at (-0.521, 0);
			\coordinate (S3F) at (0.521, 0);

			\draw[thick, blue] (-2.298, 1) -- (2.298, 1);
			\draw[thick, blue] (S2I) -- (S2F);
			\draw[thick, blue] (-2.298, 0.5) -- (0.521, 0.5);
			
			\fill (S2I) circle[radius=1.5pt];
			\fill (S2F) circle[radius=1.5pt];
			\fill (S3F) circle[radius=1.5pt];
			\fill (S1F) circle[radius=1.5pt];
			
			\node at (-1.5,0) [anchor=south]{$s_1$};
			\node at (1.5,1) [anchor=south]{$s_3$};
			\node at (0,0.5) [anchor=south]{$s_2$};
			\node at (-3,0) [anchor=south east]{$a$};
			\node at (3,0) [anchor=south west]{$b$};
			
			\node at (S2I) [anchor = north]{\tiny $ \alpha_1 \equiv \alpha_2 \equiv \alpha_3 $};
			\node at (S2F) [anchor = north]{\tiny $ \beta_1 $};
			\node at (S3F) [anchor = north]{\tiny $ \beta_2 $};
			\node at (S1F) [anchor = north]{\tiny $ \beta_3 $};
			
			\draw[dotted] (-2.298, 0) -- (-2.298, 1);
		\end{tikzpicture}
		\caption{The segments $ s_1 $, $ s_2 $ and $ s_3 $ share the same left endpoint. They are examples of segments in the class (C3).} \label{fig:C3segments}
	\end{figure}
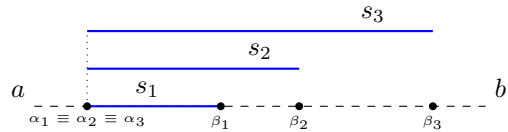
	\end{itemize}
Of course, there are families of segments that belong to more than one class. An example of this are the so-called \emph{Chebyshev} segments, that we will study in depth in Section \ref{sect:Chebyshev}. Indeed, for the particular choice of the arc-length parameter $ \rho = \pi/(2d) $ and the equidistribution of the nodes in the semicircle, $ \tau_i = \frac{(i-1/2) \pi}{d} $ for $ i = 1, \ldots, d $, they can be framed both in class (C1) and class (C2). By exploiting the normalized version of \eqref{eq:histopolation}, segments solving the Fekete problem have subsequently been recognized as an additional scenario  \cite{BEfekete}. In Section \ref{sect:classC4}, we exhibit a new class of segments which is unisolvent for polynomial histopolation. We will denote this class by (C4).

The difficulty in identifying unisolvent classes for polynomial histopolation is reflected in the lack of explicit expressions for the corresponding Lagrange bases, which are known only for the aforementioned classes (C1) \cite{GerritsmaEdge,Robidoux} and (C3) \cite{BEsinum}. The absence of closed formulas for Lagrange polynomials forces to deal with solving a linear system that involves the Vandermonde matrix $ V_d $, which is hence computed with respect to any other basis $ b_1, \ldots, b_{d} $ of $ \mathscr P_{d-1} (I) $:
\begin{equation} \label{eq:Vandermonde}
	\left[ V_d \right] _{i,j} := \int_{s_i} b_j (x) \de x .
\end{equation}
To stress the connection of the results with 
the number $ d $ of segments, 
we consistently index the Vandermonde matrix $ V_d $, as in \eqref{eq:Vandermonde}. This quantity is also the size of the matrix, and is upshifted by $1$ with respect to the relative polynomial degree.

\begin{example}
	The non-unisolvence of the set of segments introduced in the above example is also readily seen by looking at the Vandermonde matrix computed with respect to the monomial basis $ b_j (x)= x^{j-1} $. For any even $ j $, one has
		$$ \left[ V_d \right]_{i,j} = \int_{-i}^i x^{j-1} \de x = \frac{1}{j} \left[ x^j \right]_{-i}^i = \frac{1}{j} \left(i^j - i^j \right) = 0 .$$
	It follows that any even row of $ V_d $ is equal to the zero vector. Since the invertibility of the Vandermonde matrix does not depend on the polynomial basis chosen, one deduces that such a set is not unisolvent for histopolation in $ \mathscr{P}_{d-1} (I) $.
\end{example}

The problem of computing the features of a Vandermonde matrix, in particular concerning its conditioning and hence the identification of ill-conditioned bases, is a cornerstone of approximation theory \cite{Shen}; of course, so is also the opposite problem. In the framework of nodal interpolation the problem was extensively studied, and several results concerning optimality \cite{Gautschi1975, Gautschi2011, Lisimax}, bad conditioning \cite{BeckermannConditioning,BSSC}, displacement structures \cite{Displacement}, as well as their applications \cite{Li}, have been obtained, endorsing the design of efficient algorithms. The majority of these results can be grouped into two clusters: results concerning the monomial basis, which in general offers unstable performances, and results concerning the Chebyshev (of \emph{first} kind) basis, which can be used in combination with Chebyshev nodes to reverse such a situation.

\par \noindent \textbf{Outline of the paper and main results.} In contrast to nodal interpolation, very little is known about the conditioning of the Vandermonde matrix \eqref{eq:Vandermonde} associated with the histopolation problem, and about its asymptotics as a function of the polynomial degree $d$. The principal aim of this work is to fill this gap.

To do so, we dedicate Section \ref{sec:illcond} to ill-conditioning. We recognize that the (Euclidean) conditioning $ \kappa_2 $ of the matrix $ V_d $ grows exponentially with $ d $ when the monomial basis is chosen. This is proved first in the reference interval $ I = [-1,1] $ and then extended to any collection of segments by Theorem \ref{th:cheby tool3}. This feature recalls the nodal case, but is not a consequence of the mean value theorem, as one could expect. To broaden the range of applicability of the result, we also identify a new class of unisolvent segments, and use it to produce numerical evidences.

To restore hope in histopolation, in Section \ref{sect:Chebyshev} we present results concerning well-conditioning. Indeed, we show that Chebyshev polynomials of the \emph{second} kind, in combination with segments in the class (C2), give rise to a column-orthogonal matrix. The diagonal terms of $ V_d^\top V_d $ are then explicitly determined in Proposition \ref{prop:orthogonality}. As a consequence, in Theorem \ref{thm:asymptcond} we prove that the (Euclidean) conditioning $ \kappa_2 (V_d) $ is bounded, and in Theorem \ref{thm:condFrobenius} that the (Frobenius) conditioning $ \kappa_F (V_d) $ is linear. Both these results are provided with simple formulas that depend on $ \rho $.

\section{Ill-conditioning: the monomial basis} \label{sec:illcond}

The principal aim of this section is to prove that the Euclidean conditioning
$$ \kappa_2 \left( V_d \right) = \Vert V_d \Vert_2 \Vert V_d^{-1} \Vert_2 = \frac{\sigma_1 \left(V_d \right)}{\sigma_d \left(V_d \right)} = \sqrt \frac{\lambda_1 \left(V_d^\top V_d \right)}{\lambda_d \left( V_d ^\top V_d \right)} $$
of the Vandermonde matrix $ V_d $ defined in \eqref{eq:Vandermonde} is exponential when the monomial basis $ b_j (x) := x^{j-1} $ is chosen. In the above formula, as customary, $ \sigma_j (V_d) $ is the $j$-th singular value $\sigma_j^2(V_d)=\lambda_j \left(V_d^\top V_d \right) $ of $ V_d $, and the corresponding sequence
$$ \sigma_1(V_d)\ge \sigma_2(V_d)\ge\cdots \ge \sigma_d (V_d) $$
is non-increasingly ordered with $ j = 1, \ldots, d $. We will frequently set $ W_d := V_d^\top V_d $, and use both the definition and the characterization of $ \kappa_2 (V_d) $ given above.

As we shall see, the exponential conditioning is intrinsic to the monomial basis, and holds independently of the underlying collection of unisolvent segments. We build such a general result by steps, confining first ourselves in a reference interval. 
To fix the ideas we start with a motivating example concerning the class (C3), which yields very neat computations, and whose features turn out to be quite general.

\subsection{A motivating example} \label{ssec:example}

Consider $ s_i = [0, i] $, for $ i = 1, \ldots, d $, and the monomial basis. The corresponding Vandermonde matrix is
$$ \left[ V_d \right] _{i,j} = \int_{0}^i x^{j-1} \de x = \frac{i^j}{j} , $$
for which one readily computes that
$$ \left[W_d\right]_{i,j} := \left[ V_d^\top V_d \right]_{i,j} = \frac{1}{i \cdot j} \sum_{k=1}^d k^{i+j} .$$
By the well-known properties of the Raylegh quotient, for any $ x \in \R^d $ one has 
$$ \lambda_d \left( W_d  \right) \leq \frac{x^\top W_d \,x }{x^\top x} \leq \lambda_1 \left( W_d \right) .$$
By considering $ x = e_1 $, the first element of the canonical basis, one immediately gets that
$$ \lambda_d \left(W_d \right) \leq \left[W_d \right]_{1,1} = \sum_{k=1}^d k^2 = \frac{d (d+1) (2d+1)}{6} $$
and likewise, taking $ x = e_d $, one also has that
$$ \lambda_1 \left(W_d \right) \geq \left[W_d \right]_{d,d} = \frac{1}{d^2} \sum_{k=1}^d k^{2d} > d^{2d-2} .$$
It follows that, for this specific case, the conditioning of $ V_d $ is exponential with respect to the polynomial degree $d$. By the linearity of the integral, a very similar result can be obtained for the case (C1) when segments $ s_i = [i-1,i] $ are considered. These facts are pictured in Figure \ref{fig:condexample}. In what follows, we will sharpen these estimates, and remove the specific assumption on the supports.

\begin{figure}[!h]
	\centering
	\includegraphics[width=0.49\textwidth]{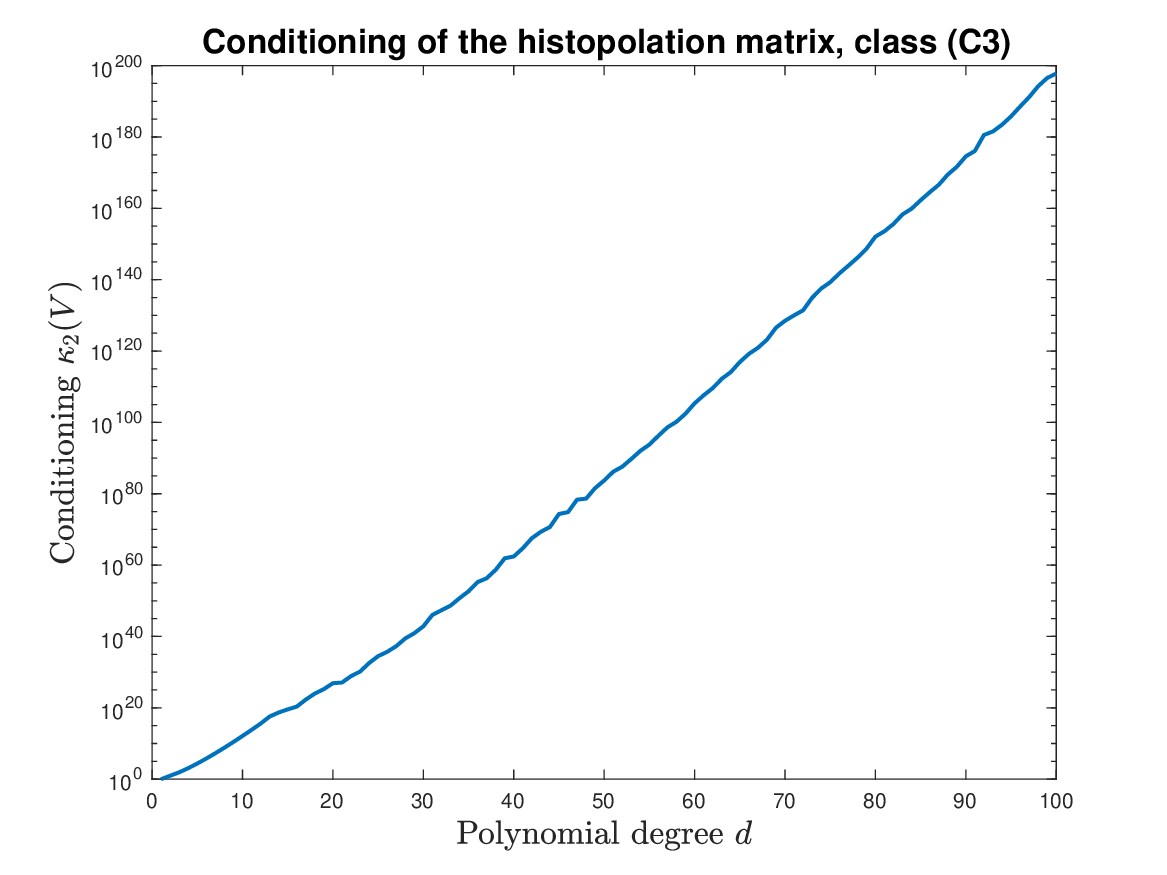}
	\includegraphics[width=0.49\textwidth]{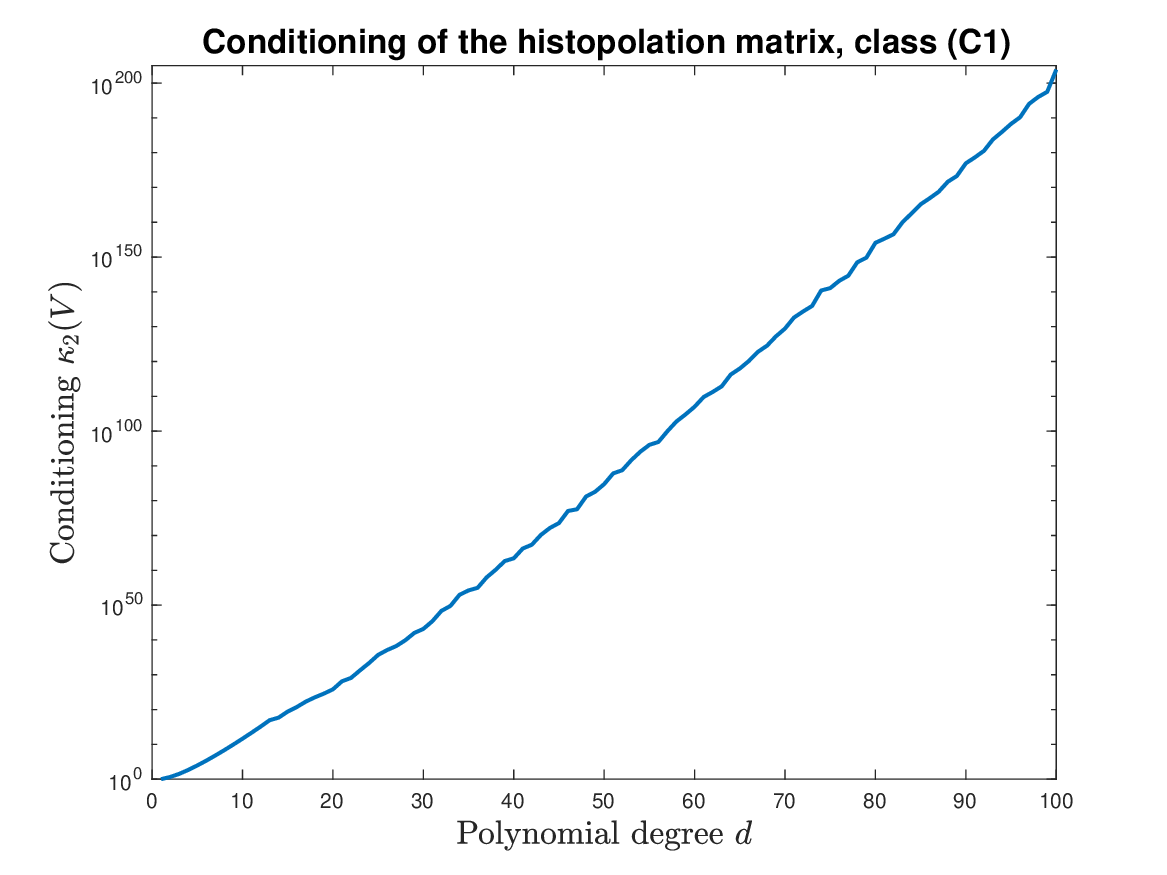}
	\caption{Exponential conditioning of the Vandermonde matrix in the monomial basis for classes (C3) and (C1), left and right hand panel respectively.} \label{fig:condexample}
\end{figure}

\begin{remark} \label{rem:reducing}
Suppose that there exists $ c > 0 $ such that $\sup_i \left| \beta_i -\alpha_i \right| \geq c$. When there exists $v>1$ independent of $d$ such that $\sup_{i,j}\max\{|\beta_i|, |\alpha_j|\}\ge v $, by direct inspection as above we have necessarily exponential growth for the maximal eigenvalue of $W_d =V_d^\top V_d$, with its minimal eigenvalue that cannot grow exponentially with respect to $d$. Hence, the latter gives a clear evidence of an exponential conditioning of $ W_d $ and a fortiori of $ V_d $ as a function of the size $d$, whenever $c>0$, $v>1$ are independent of the polynomial degree $d$. In what follows, we shall thus make a similar regularity assumption on $ c $, and discuss when it can be weakened.
\end{remark}

Observe that dividing by the size $ |s_d| = d $ of the longest support $ s_d $ one translates the problem into a fixed reference interval. Applying the change of variable $ t = x/d $, one immediately sees that the $j$-th column of $ V_d $ scales by the factor $ 1/d^j $ and, as a consequence, the $(i,j)$-th term of $ W_d $ scales by $ 1/d^{i+j} $. This shows that the conditioning of $ V_d $ depends on the length of the segments and, even if it is easy to numerically check that the conditioning in this situation is still exponential, this complicates the estimate on $ \lambda_1 $ given above, suggesting that the reference interval $ I = [-1,1] $ is somewhat a limiting case of the problem. This motivates the particular emphasis we put on this situation.

\subsubsection{A formula for the determinant for class (C3)}
The collection of segments used in Section \ref{ssec:example} is an instance of class (C3). We now extend \cite[Eq. (24)]{BEfekete} and obtain an easy formula for the determinant of $ V_d $ for the class (C3) when the left endpoint is $ \alpha_i = \alpha $. The special case $ \alpha = 0 $ clearly applies to the motivating example.

\begin{lemma} \label{lem:detVd}
	Let $ s_i = [\alpha, \beta_i] $, with $ \alpha < \beta_i < \beta_j $ for $ 1 \leq i < j \leq d $. The determinant of $ V_d $ is positive, and
	\begin{equation} \label{eq:detVd}
		\det V_d = \frac{1}{d !} \left(\prod_{i=1}^d \left( \beta_i -\alpha\right) \right) \left(\prod_{1 \leq j < k \leq d} \left(\beta_k - \beta_j \right)\right) .
	\end{equation}
\end{lemma}

\begin{proof}
	We establish formula \eqref{eq:detVd}, the claim on the positivity will be a trivial consequence of the ordering $ \alpha < \beta_i < \beta_j $ for $ 1 \leq i < j \leq d $. We compute
	$$ [V_d]_{i,j} = \int_{\alpha}^{\beta_i} x^{j-1} \de x = \frac{1}{j} \left( \beta_i^j - \alpha^j \right) ,$$
	with $ j = 1, \ldots, d $. Then $ V_d $ factorizes as $ U_d D_d $, where $ D_d $ is the diagonal matrix $ [D_d]_{j,j} = j^{-1} $ and $ U_d $ is the shifted Vandermonde matrix whose $ (i,j)$-th element is
	$$ [U_d]_{i,j} = ( \beta_i^j - \alpha^j) = (\beta_i - \alpha) \sum_{k=0}^{j-1} \alpha^k \beta_i^{j-1-k} .$$
	Gathering the term $ (\beta_i - \alpha) $ row-wise in $ U_d $, we may write $ U_d = B_d M_d $, where $ B_d $ is the diagonal matrix with entries $ [B_d]_{i,i} = (\beta_i - \alpha) $ and
	$$ \left[ M_d \right]_{i,j} := \sum_{k=0}^{j-1} \alpha^k \beta_i^{j-1-k} .$$
	By subtracting $ \alpha $ times the $j$-th column from the $ (j+1) $-th in $ M_d $, for $ j = d-1, \ldots, 1 $, one finds that
	$$ \det M_d = \det \left( \left[ \beta_i^{j-1} \right]_{i,j=1}^{d} \right)$$
	The matrix appearing at the right hand side of the above equation is the Vandermonde matrix associated with the nodes $ \{ \beta_1, \ldots, \beta_d \} $, whose determinant is well-known to be $ \prod_{1 \leq j < k \leq d} \left(\beta_k - \beta_j \right) $. The claim now follows immediately from Binet's Theorem.
\end{proof}

The formula \eqref{eq:detVd} assumes the very neat form $ \det (V_d) = \prod_{i=1}^{d-1} i ! $ under the assumptions of Section \ref{ssec:example}. This is readily seen plugging $ \alpha = 0 $ and $ \beta_i = i $ in \eqref{eq:detVd}. Lemma \ref{lem:detVd} also gives an alternative, more direct proof of the unisolvence of the class (C3) proposed in \cite{BEsinum}. Obviously, one may retrace the same proof to obtain an analogous formula by fixing the right endpoint $ \beta_i = \beta $.

\begin{corollary} \label{cor:detVd}
	Let $ s_i = [\alpha_i, \beta] $, with $ \alpha_i < \alpha_j < \beta $ for $ 1 \leq i < j \leq d $. The determinant of $ V_d $ is positive, and
	\begin{equation*} \label{eq:detVd2}
		\det V_d = \frac{1}{d !} \left(\prod_{i=1}^d \left( \beta -\alpha_i \right) \right) \left(\prod_{1 \leq j < k \leq d} \left(\alpha_k - \alpha_j \right)\right) .
	\end{equation*}
\end{corollary}

\subsubsection{An application to the Fekete problem}

The Fekete problem has been celebrated in interpolation theory since its definition \cite{Fekete}. Solutions to \emph{nodal} Fekete problems emerge in uniform approximation, as the corresponding Lagrange bases have small norm. As a consequence, the relative Lebesgue constant is small, offering good numerical stability. This holds for the \emph{segmental} Fekete problem as well \cite{BEfekete}, provided that one considers the \emph{normalized} Fekete problem: find (possibly degenerate) segments $ \{ s_1, \ldots, s_d \} $ such that
\begin{equation} \label{eq:normVanderm}
	\left| \det \left(\left[ \frac{1}{|s_i|} \int_{s_i} b_j (x) \de x \right]_{i,j}\right) \right|
\end{equation}
is maximized. Clearly, such a procedure does not depend on the polynomial basis chosen \cite{Bos}, and this fact is extensively used in \cite{BEfekete} to obtain explicit solutions for the classes (C1) and (C2). Owing to this and \eqref{eq:detVd}, we are in a position to provide a solution for the class (C3).

\begin{proposition}
	Fix $ \alpha \in I $, and let $ s_i := [\alpha, \beta_i] $. Segments $ \{ s_1, \ldots, s_d \} $ that maximize \eqref{eq:normVanderm} are such that $ \{ \beta_1, \ldots, \beta_d \} $ are the Fekete nodes of degree $ d $.
\end{proposition}

\begin{proof}
	For notational simplicity, let us assume that $ \alpha < \beta_1 < \ldots < \beta_d $. Call $ \widetilde{V}_d $ the normalized Vandermonde matrix of \eqref{eq:normVanderm}, and observe that
	$$ \det \widetilde{V}_d = \prod_{i=1}^d \frac{1}{|s_i|} \cdot \det V_d = \frac{1}{\prod_{i=1}^d \left( \beta_i - \alpha \right) } \det V_d .$$
	 Consider the monomial basis $ b_j (x) = x^{j-1} $. By \eqref{eq:detVd}, one has
	 \begin{align*} \det \widetilde{V}_d & = \frac{1}{\prod_{i=1}^d \left( \beta_i - \alpha \right) } \det V_d = \frac{1}{\prod_{i=1}^d \left( \beta_i - \alpha \right) } \frac{1}{d !} \left(\prod_{i=1}^d \left( \beta_i -\alpha\right) \right) \left(\prod_{1 \leq j < k \leq d} \left(\beta_k - \beta_j \right)\right) 
	 	\\ & = \frac{1}{d !} \prod_{1 \leq j < k \leq d} \left(\beta_k - \beta_j \right),
	 \end{align*}
	 where we also removed the absolute value of \eqref{eq:normVanderm} due to the positivity result given in Lemma \ref{lem:detVd}. Since $ \frac{1}{d!} $ is a constant term, $\det \widetilde{V}_d$ is maximized by nodes $ \{ \beta_1, \ldots, \beta_d \} $ that maximize the rightmost term of the above equality. Such nodes are well-known to be the Fekete points, which is the claim. Notice that the case $ \alpha = \beta_1 $, which represents the first segment shrinking to the point $ \alpha $, also follows by an immediate continuity argument.
\end{proof}

\subsection{A conditioning result for the reference interval} \label{ssec:main}

To extend the computations of Section \ref{ssec:example}, we start from the situation where the reference interval is $ I = [-1,1] $. This framework offers a neat handling due to the well-known features of Chebyshev polynomials \cite{ChebPoly}.

\begin{theorem}\label{th:cheby tool2}
	Let $ -1 \leq \alpha_i < \beta_i \leq 1 $ for $ i = 1, \ldots, d $, and assume that the corresponding collection of segments is unisolvent for histopolation in $ \mathscr{P}_{d-1} (I) $. Suppose that there exists $ c $ independent of $ d $ such that $ \max_{i} \left( \beta_i - \alpha_i \right) \geq c > 0$. Then
	$$ \kappa_2 \left( V_d \right) \geq \frac{c}{\sqrt{d}} 2^{d-2} .$$
\end{theorem}

\begin{proof}
	By \eqref{eq:Vandermonde}, we immediately get that
	$$ \left[V_d \right]_{i,j} = \frac{\beta_i^j-\alpha_i^{j}}{j} .$$
	Since all such terms are positive, we can trivially lower bound $ \Vert V_d \Vert_2 $ with any of the terms of $ V_d $. Hence
	$$ \Vert V_d \Vert_2 \geq \max_{1 \leq i,j \leq d} [V_d]_{i,j} \geq \max_{1 \leq i \leq d} [V_d]_{i,1} = \max_{1 \leq i \leq d} (\beta_i - \alpha_i) \geq c ,$$
	for a $ c $ independent of the polynomial degree (hence of the matrix size) $ d $. Now, in order to estimate from below the conditioning of $ V_d $ in the Euclidean norm, we look for a proper estimation from below of $ \Vert V_d^{-1} \Vert_2 $. In particular, observe that for any non-zero $ y \in \R^d $ one has
	\begin{equation} \label{eq:useful}
		\left\Vert V_d^{-1} \right\Vert_2 = \max_{\substack{z \in \R^d \\ z \ne 0}} \frac{\Vert V_d ^{-1} z \Vert_2}{\Vert z \Vert_2} \geq \frac{\Vert y \Vert_2}{\left\Vert V_d y \right\Vert_2} .
	\end{equation}
	It follows that, if we are able to select an appropriate $ y \in \R^d $, we obtain a satisfactory estimate for $ \Vert V_d^{-1} \Vert_2 $. To this end, we take $ y^\top = (a_1, \ldots, a_d) $ as the vector of coefficients of the normalized Chebyshev polynomial of the first kind and degree $ d - 1 $,
    $$ p_{d-1}^* (x) := \frac{T_{d-1} (x)}{2^{d-1}} .$$
    Due to the normalization, $ p_{d-1}^* (x) = \sum_{j=1}^d a_j x^{j-1} $ is a monic polynomial, so $ a_d = 1 $, and this implies that $ \Vert y \Vert_2 \geq | a_d | = 1 $. Further, by construction, one has
	$$ [V_d y]_{i} = \int_{\alpha_i}^{\beta_i} p^*_{d-1} (x) \de x .$$
	Since $ \Vert p_{d-1}^* \Vert_\infty = 2^{1-d} $, by the H\"older inequality one immediately gets
	$$ \left| [V_d y]_i \right| = \left| \int_{\alpha_i}^{\beta_i} p_{d-1}^* (x) \de x \right| \leq (\beta_i - \alpha_i) \Vert p_{d-1}^* \Vert_\infty \leq (\beta_i-\alpha_i) 2^{1-d} . $$
	Therefore
	$$ \Vert V_d y \Vert_2 \leq \max_{1 \leq i \leq d} (\beta_i-\alpha_i) \frac{\sqrt{d}}{2^{d-1}} \leq 2 \frac{\sqrt{d}}{2^{d-1}} = \frac{\sqrt{d}}{2^{d-2}} .$$
	We may hence estimate \eqref{eq:useful} from below by $ \Vert V_d^{-1} \Vert_2 \geq 1/\Vert V_d y \Vert_2 \geq 2^{d-2} / \sqrt{d} $. Gathering all the results, we have
	$$ \kappa_2 \left( V_d \right) = \Vert V_d \Vert_2 \Vert V_d^{-1} \Vert_2 \geq \frac{c}{\sqrt{d}} 2^{d-2} .$$
	Since $ c $ does not depend on the degree $ d $, the claim is proved.
\end{proof}

Theorem \ref{th:cheby tool2} applies to the class of segments (C1) where $ 1 < \beta_i = \alpha_{i+1} < 1 $ for $ i = 1, \ldots, d-1 $. It contains also the class of segments (C3) for $ 1 \leq \alpha = \alpha_i $ for $ i = 1, \ldots, d $ and $ \beta_i \ne \beta_j $ for $ i \ne j $. Even if confined to the reference element $I = [-1, 1] $, Theorem \ref{th:cheby tool2} proves that the trend depicted in Figure \ref{fig:condexample} is not peculiar to the selection of segments performed in the motivating Section \ref{ssec:example}. The subsequent Theorem \ref{th:cheby tool3} will in fact demonstrate that this is intrinsic to the monomial basis.

\subsubsection{The normalized case}

In the determination of Fekete segments, the presence of a normalization factor has proven to play an important role \cite{BEfekete}, giving a hybrid situation between the limit where all segments shrink to points and the technical request $ \sup_{i} (\beta_i - \alpha_i) \geq c > 0 $. From the perspective of the conditioning of the associated Vandermonde matrix
$$ [\widetilde V_d]_{i,j} := \frac{1}{\beta_i - \alpha_i} \int_{\alpha_i}^{\beta_i} x^{j-1} \de x $$
this does not mark a relevant difference, as the next result shows.

\begin{corollary} \label{cor:normalizedC1}
	Let $ -1 \leq \alpha_i < \beta_i \leq 1 $ for $ i = 1, \ldots, d $, and assume that the corresponding collection of segments is unisolvent for histopolation in $ \mathscr{P}_{d-1} (I) $. Then
	$$ \kappa_2 \left( \widetilde V_d \right) \geq \frac{2^{d-1}}{\sqrt{d}} .$$    
\end{corollary}

\begin{proof}
    Following exactly the same lines of the proof of Theorem \ref{th:cheby tool2}, we have
    $$ \Vert \widetilde V_d \Vert_2 \geq \max_{1 \leq i,j \leq d} [\widetilde V_d]_{i,j} \geq \max_{1 \leq i \leq d} [\widetilde V_d]_{i,1} = \max_{1 \leq i \leq d} \frac{\beta_i - \alpha_i}{\beta_i - \alpha_i} =1 .$$
    Likewise, the estimate using normalized Chebyshev polynomials gives
    $$ \left| [\widetilde V_d y]_i \right| = \left| \frac{1}{\beta_i - \alpha_i}\int_{\alpha_i}^{\beta_i} p_{d-1}^* (x) \de x \right| \leq \frac{\beta_i - \alpha_i}{\beta_i - \alpha_i} \Vert p_{d-1}^* \Vert_\infty \leq 2^{1-d} . $$
    Hence, we conclude that $ \kappa_2 \left( \widetilde V_d \right) \geq \frac{2^{d-1}}{\sqrt{d}} $, improving the above result by the factor $ c/2 $.
\end{proof}

The subsequent extension of Theorem \ref{th:cheby tool2} to generic intervals will immediately carry over to the normalized case. Note that Corollary \ref{cor:normalizedC1} does not require anymore the assumption $ \sup_{i} (\beta_i - \alpha_i) \geq c > 0 $, and is in fact blind to $ c $. In this case, the histopolation problem can be properly defined on \emph{regular} sets \cite{Robidoux}, and involves both averages and point evaluations.

\subsection{Segments with uniform length: the class (C4)} \label{sect:classC4}

The technical assumption $ \max_i (\beta_i - \alpha_i) \geq c > 0 $ makes the result of Theorem \ref{th:cheby tool2} tailored for a new class of segments, in which the constant $ c $ is the measure of a given reference element.

\begin{definition}
	Let $ s = [\alpha, \beta] $ be a reference segment and let $ \{ \xi_1, \ldots, \xi_d \} $ be real nodes. We define the segments in the class (C4) as the collection $ \{ s_1, \ldots, s_d \} $, where
	$$ s_i := \{ x + \xi_i : \ x \in s \} .$$
	Clearly, one has $ s_i = [\alpha + \xi_i, \beta + \xi_i ] $ and $ | s_i | = | s | $ for each $ i = 1, \ldots, d $.
\end{definition}

\begin{figure}[h]
	\centering	
	\begin{tikzpicture}[baseline=(current bounding box.north)]
		\coordinate (A) at (-3,0);
		\coordinate (B) at (3,0);
		\coordinate (C) at (1.5, 2.598);
		\coordinate (D) at (-1.5, 2.598);
		\coordinate (O) at (0,0);
		
		\draw[dashed] (A) -- (B);

		\coordinate (S2I) at (-2.298, 0);
		\coordinate (S2F) at (-0.521, 0);
		
		\coordinate (S1I) at (-1.298, 0);
		\coordinate (S1F) at (0.479, 0);
		
		\coordinate (S3I) at (-0.298, 0);
		\coordinate (S3F) at (1.479, 0);

		\draw[thick, blue] (-1.298, 0.5) -- (0.479, 0.5);
		\draw[thick, blue] (S2I) -- (S2F);
		\draw[thick, blue] (-0.298, 1) -- (1.479, 1);
		
		\fill (S2I) circle[radius=1.5pt];
		\fill (S1I) circle[radius=1.5pt];
		\fill (S3I) circle[radius=1.5pt];
		
		\node at (-1.5,0) [anchor=south]{$s_1$};
		\node at (1.5,1) [anchor=south]{$s_3$};
		\node at (0,0.5) [anchor=south]{$s_2$};
		\node at (-3,0) [anchor=south east]{$a$};
		\node at (3,0) [anchor=south west]{$b$};
		
		\draw[dotted] (-1.298, 0) -- (-1.298, 0.5);
		\draw[dotted] (-0.298, 0) -- (-0.298, 1);
		\node at (-2.298, 0) [anchor = north]{\tiny $ \xi_1$};
		\node at (-1.298, 0) [anchor = north]{\tiny $ \xi_2 \ne \xi_1 $};
		\node at (-0.298, 0) [anchor = north]{\tiny $ \xi_3 \ne \xi_{1,2} $};
	\end{tikzpicture}
	\caption{The segments $ s_1 $, $ s_2 $ and $ s_3 $ are translated of the same segment. They are examples of segments in the class (C4).} \label{fig:C4segments}
\end{figure}
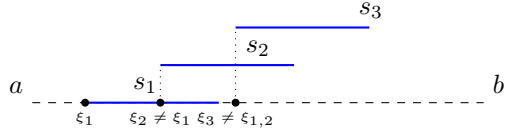

The class (C4) shows that the way segments overlap is a rather weak information for the purpose of unisolvence. In fact, we are going to show that the class (C4) is unisolvent for histopolation in $ \mathscr P_{d-1} (I) $ whenever the nodes $ \xi_i $, $ i = 1,\ldots,d$, are pairwise distinct (i.e., unisolvent for nodal interpolation), with no restriction on the measure of $ s $. This can be seen as a consequence of the following general fact.

\begin{lemma} \label{lem:FT}
	Let $ p : \R \to \R $ be a polynomial, and let $ s $ be an interval. Define the translated of $ s $ as $ s + \xi := \{ x + \xi : \ x \in s \} $. If
	$$ \int_{s+\xi} p(x) \de x = 0 $$
	for each $ \xi \in \R $, then $ p (x) = 0 $ for each $ x \in \R $.
\end{lemma}

\begin{proof}
 Let $ P $ be any antiderivative of $ p $. Obviously, since $ p $ is a polynomial, so is $ P $. Then, writing the reference segment as $ s = [\alpha, \beta] $, one has
 $$ 0 = \int_{s + \xi} p(x) \de x = \int_{\alpha + \xi}^{\beta +\xi} p(x) \de x = P(\beta+\xi) - P(\alpha+\xi) ,$$
 which implies that $ P(\beta+\xi) = P(\alpha + \xi) $ for each $ \xi \in \R $. As a consequence, $ P $ is a periodic function of period $ | s | = \beta - \alpha $, and since it is a polynomial, it must be the constant polynomial. Since $ P $ is an antiderivative of $ p $, one has $ p = P' $, and since $ P $ is constant it follows that $ p (x) = P'(x) = 0 $ for each $ x \in \R $.
\end{proof}

We are in a position to prove the following unisolvence result.

\begin{proposition} \label{prop:unisolvence}
	Let $ s $ be a segment, and define $ s + \xi := \{ x + \xi : \ x \in s \} $. Consider $ \left\{ \xi_1, \ldots, \xi_{d} \right\} $ pairwise distinct nodes and let $ p_{d-1} \in \mathscr{P}_{d-1} (I) $. If
	$$ \int_{s + \xi_i} p_{d-1}(x) \de x = 0 $$
	for each $ i = 1, \ldots, d $, then $ p_{d-1} (x) = 0 $ for each $ x \in \R $.
\end{proposition}

\begin{proof}
	Consider the map
	$$ \xi \mapsto \int_{s+ \xi} p_{d-1} (x) \de x . $$
	By the change of variable, one sees that this is a polynomial $ q_{d-1}(\xi) $ of degree $ d - 1 $ in the variable $ \xi $. By hypothesis, this polynomial vanishes on the grid $ \left\{ \xi_1, \ldots, \xi_{d} \right\} $, which consists of distinct nodes. Hence the polynomial $ q_{d-1}(\xi) $ vanishes on the whole real line $ \R $, implying that $$ \int_{s + \xi} p_{d-1} (x) \de x = 0 $$ 
	for each $ \xi \in \R $. By Lemma \ref{lem:FT}, one gets $ p_{d-1} (x) = 0 $ for each $ x \in \R $.
\end{proof}

\begin{remark}
	Proposition \ref{prop:unisolvence} also has a different proof. 
	To this end, observe that the set of zeros of $ p_{d-1} (x) $ is a discrete set, containing at most $ d-1 $ nodes. If it had more, then $ p_{d-1} $ would be the zero polynomial. Call $ \xi_{\max} $ the largest one, i.e. the node $ \xi_i $ such that $ \xi_i \geq \xi_j $ for any $ j = 1, \ldots, d-1 $. Then, the result of Lemma \ref{lem:FT} can be deduced as there exists $ \xi \in \R $ such that $ \xi - |s| \geq \xi_{\max} $. Hence, the sign of $ p_{d-1} $ is constant in the interior of $ s + \xi $, and since $ \int_{s+\xi} p_{d-1} (x) \de x = 0 $, one has $ p_{d-1} (x) = 0 $ for each $ x $ in the (non empty) interior of $ s + \xi $. Since the only polynomial vanisihing in an open and non empty set is the zero polynomial, this implies that $ p_{d-1} (x) = 0 $ for each $ x \in \R $.
\end{remark}

The class (C4) has intersection with the class (C1), which is given by what we called, in the introduction, equidistributed segments. 
Due to Theorem \ref{th:cheby tool2}, The Vandermonde matrix associated with such segments and the monomial basis presents asymptotically an exponential ill-conditioning as a function of $d$. This is numerically verified and shown in Figure \ref{fig:exponentialclassC4}.

\begin{figure}[h!]
	\centering
	\includegraphics[width=0.75\textwidth]{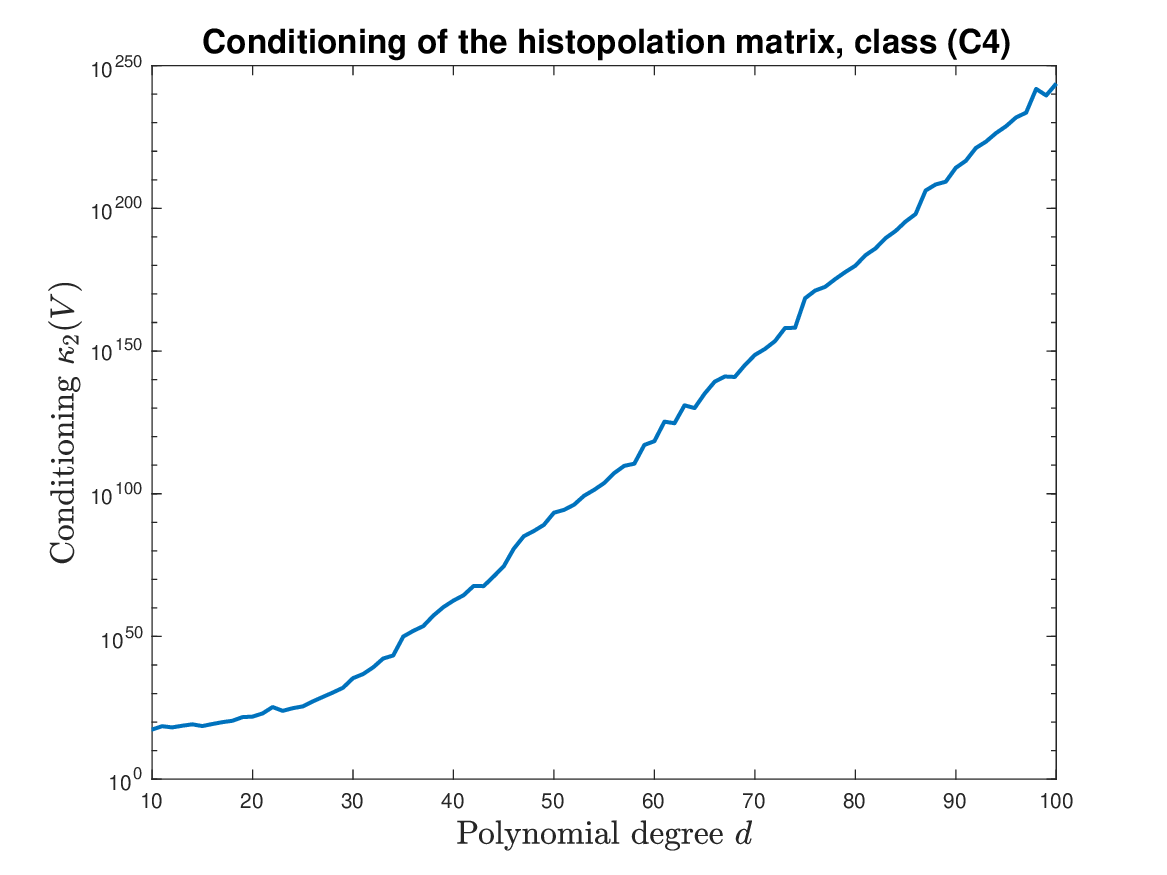}
	\caption{Exponential asymptotic ill-conditioning for segments in the class (C4).} \label{fig:exponentialclassC4}
\end{figure}

\subsection{Extension to generic intervals} \label{sect:generic}

The extension of Theorem \ref{th:cheby tool2} to the generic case is based on the following result.

\begin{theorem}\label{th:cheby tool3}
	Let $\alpha_i < \beta_i $ for $ i = 1, \ldots, d $ and assume $\liminf_{d\rightarrow \infty} \min_i \alpha_i=-1$ or, equivalently, $\limsup_{d\rightarrow \infty} \max_i\beta_i=1$. Suppose that there exists $ c $ independent of $ d $ such that $ \max_i (\beta_i - \alpha_i) \geq c > 0 $. Then 
	$$ \kappa_2 \left( V_d \right) \geq \frac{c}{\sqrt{d}} 2^{d-2} .$$
\end{theorem}

\begin{proof}
We assume $\limsup_{d\rightarrow \infty} \max_i\beta_i=1$ and let $i^*_d$ such that $\beta_{i^*_d}=\max_{1\le i \le d} \beta_i $. With no loss of generality we can assume $\alpha_i\ge -1$. Hence
\[
	\frac{\beta_i^j-\alpha_i^j}{j}=\beta_{i^*_d}^j \frac{\gamma_i^j-\delta_i^j}{j},
\]
$\gamma_i=\frac{\beta_i}{\beta_{i^*_d}}\le 1$, $\delta_i=\frac{\alpha_i}{\beta_{i^*_d}}\ge -1$.
Hence $V_d=\widehat V_dD_d$ with $D_d=$diag$_j(\beta_{i^*_d}^j)$ so that we deduce
\[
\kappa_2(\widehat V_d)\le \kappa_2(V_d)\kappa_2(D_d)
\]
i.e.
\[
\kappa_2(V_d)\ge \frac{\kappa_2(\widehat V_d)}{\kappa_2(D_d)}.
\]
Now, taking into account the hypothesis $\limsup_{d\rightarrow \infty} \max_i\beta_i=1$, we infer that for each $ \varepsilon>0$, there exists $ d_\varepsilon$ such $\kappa_2(D_d)\le (1+\varepsilon)^d$, for any $d\ge d_\varepsilon$ and finally by Theorem \ref{th:cheby tool2} we conclude the proof since
\[
\kappa_2(V_d)\ge \frac{\kappa_2(\widehat V_d)}{\kappa_2(D_d)} \ge \frac{c}{\sqrt{d}} \frac{2^{d-2}}{(1+\varepsilon)^d}
\]
and $\varepsilon$ is chosen so that $1+\varepsilon<2$.
\end{proof}

To better frame Theorem \ref{th:cheby tool3}, we conclude the section gathering some considerations and consequences.

\begin{remark} \label{rem:reducing-bis}
In connection with the proof of Theorem \ref{th:cheby tool3}, we note that the latter is the missing piece of the picture. Indeed, we are finally in the position of completing the reasoning with the following observation: if $\liminf_{d\rightarrow \infty}\lim_i \alpha_i<-1$ or $\limsup_{d\rightarrow \infty} \max_i \beta_i>1$, then passing to proper subsequences, we go back to the same setting as in Remark \ref{rem:reducing}.  As a conclusion the exponential conditioning is a general phenomenon. This is particularly interesting given the link of the considered structures with the Vandermonde matrices on with nodes on the real line; see \cite{BeckermannConditioning,SSCconditioning}.
\end{remark}

\begin{remark}
If we consider standard induced norms $\|\cdot\|_{l^k\rightarrow l^p}$, $k,p\in [1,\infty]$ (the case $k=p=2$ being the standard Euclidean case, also known as spectral norm), then we obtain again exponential conditioning, since the equivalence constants among all these induced norms have a rational expression with respect to the variable $d$.
\end{remark}

\begin{remark} \label{rem:further-generalizing}
The picture given in Theorems \ref{th:cheby tool2} and \ref{th:cheby tool3} and in Remarks \ref{rem:reducing} and \ref{rem:reducing-bis} is even more general. In fact, if we consider sets $\{i \cdot \alpha_i^{(d)} =1,\ldots,d\}$, $\{ i \cdot \beta_i^{(d)} =1,\ldots,d\}$, of natural numbers, in which for $d_1\neq d_2$ we have in general $\alpha_i^{(d_1)}\neq \alpha_i^{(d_2)}$ and $\beta_i^{(d_1)}\neq \beta_i^{(d_2)}$, $i\leq \min \{d_1,d_2\}$, all the proofs work unchanged. This represents a further degree of generality in the analysis.
\end{remark}

\subsection{Some remarks on the independence of $ c $ on $ d $}

In the non-normalized case, we have assumed the existence of a \textquotedblleft separation distance\textquotedblright\ $ c > 0 $ such that $\max_i \left(\beta_i -\alpha_i \right) \geq c $. We asked such a constant to be independent of $ d $. This is clearly motivated by the fact that, if $ c $ depended on $ d $, it could decay to zero fast enough to simplify the exponential term $ 2^{d-2} $. It is plain to observe that, if $ c $ depends on $ d $ and does not decay exponentially, the exponential trend of the conditioning of the monomial basis persists. An example is provided when considering equidistributed segments in $ I = [-1, 1] $.

\begin{example}
	Let $ I $ be an interval, and let $ s_1, \ldots, s_d $ be equidistributed segments covering $ I $. Clearly, one has $ | s_i | = \beta_i - \alpha_i = |I|/d $ for each $ i = 1, \ldots, d $, and this quantity tends to zero linearly with $ d $ as $ d $ grows, contradicting the hypotheses on $ c $ required in the statement of Theorem \ref{th:cheby tool2}. However, the conditioning is still exponential.
	
	Since
	$$ \Vert V_d \Vert_2 = \max_{\substack{z \in \R^d \\ z \ne 0}} \frac{\Vert V_d z \Vert_2}{\Vert z \Vert_2}, $$
	by taking $ z = e_1 $, the first element of the canonical basis of $ \R^d $, one has
	$$ \Vert V_d \Vert_2 \geq \Vert V_d e_1 \Vert_2 = \sqrt{\sum_{i=1}^d ( \beta_i - \alpha_i )^2}  = \sqrt{\sum_{i=1}^d \left( \frac{|I|}{d} \right)^2} = \sqrt{d \left( \frac{|I|}{d} \right)^2} = \frac{|I|}{\sqrt{d}} .$$
	Merging this with the estimate for $ \Vert V_d^{-1} \Vert_2 $ given in Theorem \ref{th:cheby tool2}, which does not depend on the size of the segments, we get
	$$ \kappa_2 (V_d) = \Vert V_d \Vert_2 \Vert V_d^{-1} \Vert_2 \geq \frac{|I|}{\sqrt{d}} \frac{2^{d-1}}{\sqrt{d}} = |I| \frac{2^{d-2}}{d} . $$
	Notice that this represents a limiting value for all collection of segments that are arranged into chains, and that this result does not require specific hypotheses on $ I $.
\end{example}

Suppose, in complete contrast, that such a constant $ c = \max_{i} (\beta_i - \alpha_i) $ is equal to zero; this means that every segment $ s_i $ collapses to a point $ \xi_i $. Even if these nodes are distinct, the problem \eqref{eq:histopolation} becomes ill-posed, as the matrix \eqref{eq:Vandermonde} becomes the zero matrix. This situation can be handled by placing the normalization factor $ |s_i| $ in the histopolation conditions and considering the normalized case. If one imposes that the limiting nodes $ \{ \xi_1, \ldots, \xi_d \} $ are pairwise distinct, by an immediate continuity argument and the mean value theorem, one deduces that the problem has a unique solution \cite{BEfekete}, and Corollary \ref{cor:normalizedC1} implies that the corresponding matrix is exponentially ill-conditioned. However, the same continuity argument shows that the Vandermonde matrix \eqref{eq:Vandermonde} converges (in an entry-wise sense) to the nodal Vandermonde matrix based upon the nodes $ \{ \xi_1, \ldots, \xi_d \} $ and written with respect to the monomial basis. The well-established results concerning the ill-conditioning of the monomial basis for the nodal theory \cite{BeckermannConditioning,SSCconditioning} therefore give a different proof of the exponential ill-conditioning in this limit.

\section{Well-conditioning: Chebyshev basis and segments} \label{sect:Chebyshev}

In stark contrast to Section \ref{sec:illcond}, a very different situation appears when the \emph{Chebyshev polynomials of the second kind}
$$ U_{j-1} (x) = \frac{\sin (j \arccos x)}{\sin (\arccos x)} $$
are considered. This family is not the most diffused in nodal interpolation, where the \emph{first} one
$$ T_j (x) = \cos(j \arccos x) $$
is usually adopted. The two families are related via differentiation, as $ \de / \de x \left( T_j (x) \right) = j U_{j-1} (x) $. As computed in \cite[Section $2.2$]{BEsinum}, the Vandermonde matrix \eqref{eq:Vandermonde} assumes a very peculiar form with respect to the basis $ U_{j-1} (x) $:
\begin{align} \label{eq:VandermondeU}
	\left[V_d \right]_{i,j} & = \int_{\alpha_i}^{\beta_i} U_{j-1}(x) \de x = \frac{1}{j} \int_{\alpha_i}^{\beta_i} \frac{\de T_{j}(x)}{\de x} \de x = \frac{1}{j} \left( T_j (\beta_i) - T_j (\alpha_i) \right) \\
	& = \frac{2}{j} \sin \left( j \frac{\arccos \alpha_i +\arccos \beta_i}{2}\right) \sin \left( j \frac{\arccos \alpha_i -\arccos \beta_i}{2}\right) . \nonumber
\end{align}

We show that this basis is particular well conditioned when considered in combination with segments in the class (C2). Recall that these segments are defined as
\begin{equation} \label{eq:classC2}
	s_i = [\alpha_i, \beta_i] = [\cos(\tau_i + \rho), \cos (\tau_i - \rho) ] ,
\end{equation}
being $ \tau_i $ the (pre-image of the) $i$-th Chebyshev node $ \tau_i = \frac{(i-1/2) \pi}{d} $ for $ i = 1, \ldots, d $ and $ \rho > 0 $ an arc length veryfing the existence conditions established in \cite[Section 3.2]{BEsinum}. Plugging the definition of the polynomial $ U_{j-1} (x) $ and class (C2) segments into \eqref{eq:VandermondeU}, one immediately obtains that the respective Vandermonde matrix is 
\begin{equation} \label{eq:vandermondeC2}
	\left[V_d \right]_{i,j} = \frac{2}{j} \sin (j \tau_i) \sin (j \rho) .
\end{equation}
This matrix has orthogonal columns.

\begin{proposition} \label{prop:orthogonality}
	The matrix $ V_d^\top V_d $ is diagonal. The $i$-th diagonal term is
	\begin{equation} \label{eq:diagonalterm}
		\left[ V_d^\top V_d \right]_{i,i} = \begin{cases} 
			\frac{2d}{i^2} \sin^2 (i \rho) & \text{if } i = 1, \ldots, d-1, \\
			\frac{4d}{i^2} \sin^2 (i \rho) & \text{if } i = d .
			\end{cases}
	\end{equation}
	Notice that, when $ i = d $, the diagonal term is $ 4d/i^2 \sin^2 (i \rho) = 4/d \sin^2 (d \rho) $.
\end{proposition}

\begin{proof}
	By direct computation, one obtains
	\begin{align} \label{eq:diagVtV}
		\left[ V_d^\top V_d \right]_{i,j} & = \sum_{k=1}^d \left[V_d^\top \right]_{i,k} \left[V_d \right]_{k,j} = \sum_{k=1}^d \left[V_d\right]_{k,i} \left[V_d\right]_{k,j}
		 \\ & = \sum_{k=1}^d \frac{2}{i} \sin (i \tau_k) \sin (i \rho) \frac{2}{j} \sin (j \tau_k) \sin (j \rho) \nonumber
		= \frac{4 \sin (i \rho) \sin (j \rho)}{ij} \sum_{k=1}^d \sin (i \tau_k) \sin(j \tau_k) .
	\end{align}
    To conclude the proof, we shall establish the discrete orthogonality
    	\begin{equation} \label{eq:orthsine}
	\sum_{k=1}^d \sin (l \tau_k) \sin(m \tau_k) = 
	\begin{cases}
		\frac{d}{2} & \text{if } l = m < d \\
		d & \text{if } l = m = d \\
        0 & \text{if } l \ne m.
	\end{cases} 
	\end{equation}
	To exploit the features of the Chebyshev nodes, we use Werner's formula
	$$ \sin (l \tau_k) \sin(m \tau_k) = \frac{1}{2} \left[ \cos((l-m) \tau_k) - \cos ((l+m) \tau_k) \right] .$$
	Consider first the case $ l = m $, i.e. the diagonal terms of $ V_d^\top V_d $. One has
	\begin{align} \label{eq:Werner}
		\sum_{k=1}^d \sin (l \tau_k) \sin(l \tau_k) & = \frac{1}{2} \sum_{k=1}^d \left[ \cos((l-l) \tau_k) - \cos ((l+l) \tau_k) \right] = \frac{1}{2} \sum_{k=1}^d \left[ \cos(0) - \cos (2 l \tau_k) \right] \nonumber \\
		& = \frac{1}{2} \sum_{k=1}^d 1 - \frac{1}{2} \sum_{k=1}^d \cos (2l \tau_k) = \frac{d}{2} - \frac{1}{2} \sum_{k=1}^d \cos (2l \tau_k) .
	\end{align}
    Recalling that $ \tau_k = (2k-1)\pi/(2d) $, 
    we compute the sum $ \sum_{k=1}^d \cos (2l \tau_k) $. Denoting by $ \Re $ the real part of a complex number and by $ \iota $ the imaginary unit, by Euler's formula one has
	\begin{align*} 
    \sum_{k=1}^d \cos \left( \frac{l \pi}{d} (2k-1) \right) & = \sum_{k=1}^d \Re \left( e^{\iota \frac{l \pi}{d} (2k-1)} \right) = \Re \left(  \sum_{k=1}^d e^{\iota \frac{l \pi}{d} (2k-1)} \right) \\
    & = \Re \left( e^{- \iota \frac{l \pi}{d}} \sum_{k=1}^d e^{\iota \frac{l \pi 2k}{d}}\right) = \Re \left( e^{- \iota \frac{l \pi}{d}} \sum_{k=1}^d \left( e^{\iota \frac{l 2 \pi}{d}} \right)^k \right) .
    \end{align*}
	For $ l = d $, one obtains $ e^{i \frac{l 2 \pi}{d}} = e^{i \frac{d 2 \pi}{d}} = e^{i 2 \pi} = \cos(2 \pi) = 1 $. Hence the real part can be removed:
	$$ \sum_{k=1}^d e^{\iota \frac{d \pi}{d} (2k-1)} = e^{- \iota \pi} \sum_{k=1}^d 1^k = - d . $$
    Plugging this into \eqref{eq:Werner} one obtains the claim for $ l = m = d $.

	The case $ 1 \leq l < d $ can be treated as a geometric series of ratio $ e^{\iota \frac{l 2 \pi}{d}} $, which gives 
	\begin{equation} \label{eq:trigonometriclong}
		e^{- \iota \frac{l \pi}{d}} \sum_{k=1}^d \left( e^{\iota \frac{l 2 \pi}{d}} \right)^k  = e^{- \iota \frac{l \pi}{d}} e^{\iota \frac{l 2 \pi}{d}}  \frac{1-\left(e^{\iota \frac{l 2 \pi}{d}} \right)^d}{1-e^{\iota \frac{l 2 \pi}{d}} } = e^{- \iota \frac{l \pi}{d}} e^{\iota \frac{l 2 \pi}{d}}  \frac{1-e^{\iota l 2 \pi}}{1-e^{\iota \frac{l 2 \pi}{d}} } .
	\end{equation}
	Since $ 1 \leq l < d $, $ 2 \pi l $ is an integer multiple of $ 2 \pi $, while $ 2 \pi l / d$ is not, and thus $ e^{i \frac{l 2 \pi}{d}} \ne 1 $ for any $ l = 1, \ldots, d-1 $. It follows that 
	$$ \frac{1-e^{\iota l 2 \pi}}{1-e^{\iota \frac{l 2 \pi}{d}} } = \frac{1-1}{1-e^{\iota \frac{l 2 \pi}{d}} } =  0 .$$
    Plugging this into \eqref{eq:Werner} we retrieve the first equality of \eqref{eq:orthsine}. Therefore, the claim on the diagonal terms of $ V_d^\top V_d $ is established. 
	
	We are left to prove that the off-diagonal terms of $ V_d^\top V_d $ vanish. 
    We hence show that 
	$$ \sum_{k=1}^d \cos \left( (l-m) \tau_k \right) - \sum_{k=1}^d \cos \left( (l+m) \tau_k \right) = 0 .$$
	To prove the above equality, we shall use the trigonometric representation of complex numbers again. Owing to the linearity of the real part, we focus on 
	$$ \sum_{k=1}^d \cos\left( \frac{n (2k-1) \pi}{d} \right) = \Re \left( \sum_{k=1}^d e^{\iota \frac{n}{2d} (2k-1) \pi } \right) ,$$
	either for $ n = l-m $ or $ n = l+m $. We are going to show that both sums are equal to zero (when both $ l - m $ and $ l + m $ are even), or that both sums add to a completely imaginary number (when both $ l - m $ and $ l + m $ are odd). This reasoning clearly exhausts all cases.
	
	Following the same reasoning that lead to \eqref{eq:trigonometriclong}, one finds
	\begin{equation} \label{eq:lasttrigonometric}
		\sum_{k=1}^d e^{\iota \frac{n}{2d} (2k-1) \pi } = e^{-\iota \frac{n \pi}{2d}} \sum_{k=1}^d \left( e^{\iota \frac{n \pi}{d}} \right)^k = e^{-\iota \frac{n \pi}{2d}} e^{\iota \frac{n \pi}{d}} \frac{1- e^{\iota n \pi}}{1- e^{\iota \frac{n \pi}{d}}} .
		\end{equation}
	The rightmost hand side of \eqref{eq:lasttrigonometric} exists as $ n $ cannot be an entire multiple of $ 2d $ (because $ l $ and $ m $ represent off-diagonal terms, and so $ - d < l-m < d $ with $ l-m \ne 0 $ and $ 0 < l + m < 2d $), and this implies that $ e^{\iota \frac{n \pi}{d}} \ne 1 $.
	
	To conclude the proof we consider separately the even and the odd case. First, 
	suppose that $ n = l-m $ is even. In this case $ e^{\iota n \pi} = \cos(2 \pi) + \iota \sin(2 \pi) = 0 $, and as a consequence
	$$ \sum_{k=1}^d e^{\iota \frac{n}{2d} (2k-1) \pi } = e^{-\iota \frac{n \pi}{2d}} e^{\iota \frac{n \pi}{d}} \frac{1- e^{\iota n \pi}}{1- e^{\iota \frac{n \pi}{d}}} = 
    e^{\iota \frac{n \pi}{2d}} \frac{1 - 1}{1 - e^{\iota \frac{n \pi}{d}}} = 0 . $$
	Assume then $ n = l-m $ is odd. In such a case, $ e^{\iota n \pi} = \cos (\pi) + \iota \sin (\pi) = -1 $. Treating the sum as in \eqref{eq:lasttrigonometric} and manipulating, one gets
	\begin{align*}
		\sum_{k=1}^d e^{\iota \frac{n}{2d} (2k-1) \pi } & = e^{- \iota \frac{n \pi}{2d}} e^{\iota \frac{n \pi}{d}} \frac{1- e^{\iota n \pi}}{1- e^{\iota \frac{n \pi}{d}}} = 
        e^{\iota \frac{n \pi}{2d}} \frac{1 + 1}{1- e^{\iota \frac{n \pi}{d}}} = 2 \frac{e^{\iota n \frac{\pi}{2d}}}{1-e^{\iota n \frac{\pi}{d}}} = \frac{2}{e^{-\iota n \frac{\pi}{2d}}-e^{\iota n \frac{\pi}{2d}}} .
    \end{align*}
Substituting the identity $ e^{\iota n \frac{\pi}{2d}} - e^{- \iota n \frac{\pi}{2d}} = 2 \iota \sin (\frac{n \pi}{2d}) $ and simplifying, one concludes that 
$$ \sum_{k=1}^d e^{\iota \frac{n}{2d} (2k-1) \pi } = \frac{\iota}{\sin(\frac{n \pi}{2d}) }, $$
	which has no real part. The same argument applies to $ \sum_{k=1}^d \cos \left((l+m) \tau_k\right) $, whence it follows that
	$$  \Re \left( \sum_{k=1}^d \cos \left( (l-m) \tau_k \right) \right) = \Re \left( \sum_{k=1}^d \cos \left( (l+m) \tau_k \right) \right) = 0. $$
	This immediately implies the last equality of \eqref{eq:orthsine}. 
    Gathering all the computations and plugging them into \eqref{eq:diagVtV}, one obtains
	\begin{equation*}
		\left[ V_d^\top V_d \right]_{i,j} = 
		\begin{cases}
			\frac{2d}{i^2} \sin^2 (i \rho) & \text{if } i = 1, \ldots, d-1, \\
			\frac{4d}{i^2} \sin^2 (i \rho) & \text{if } i = d, \\
			0 & \text{if } i \ne j, 
		\end{cases}
	\end{equation*}
	which shows the claim.
\end{proof}

Proposition \ref{prop:orthogonality} is the key ingredient to unveil some features of $ V_d $ and the corresponding matrix sequence. The first easy consequence is an exact formula for the modulus of the determinant of $ V_d $, when the class (C2) is considered in combination with Chebyshev polynomials of the second kind. The sign of $ \det V_d $ depends on $ d $, and due to the reverse ordering of the Chebyshev nodes it equals $ (-1)^{\frac{d(d-1)}{2}}$.

\begin{corollary} \label{cor:detC2}
	One has
	\begin{equation} \label{eq:detVdCheby}
		\left| \det V_d \right| = \sqrt{2 \frac{(2d)^d}{(d!)^2} \prod_{i=1}^d \sin^2 \left( i \rho \right)} .
	\end{equation}
\end{corollary}

\begin{proof}
	It is clear that $ \det (V_d^\top V_d) = \left(\det V_d\right)^2 $, whence
	\begin{align*} 
		\left| \det V_d \right| & = \sqrt{ \det (V_d^\top V_d) } = \sqrt{\prod_{i=1}^d [V_d^\top V_d]_{i,i} } = \sqrt{\left( \prod_{i=1}^{d-1} \frac{2d}{i^2} \sin^2 (i \rho) \right) \frac{4d}{d^2} \sin^2 (d \rho)} 
		\\ & = \sqrt{2 \prod_{i=1}^{d} \frac{2d}{i^2} \sin^2 (i \rho)}.
	\end{align*}
	Let us examine the argument of the last square root, up to the factor $ 2 $. One has
	\begin{equation} \label{eq:partialprod}
		\prod_{i=1}^d \frac{2d}{i^2} \sin^2 (i \rho) = \frac{\prod_{i=1}^d 2d}{\prod_{i=1}^d i^2} \prod_{i=1}^d \sin^2 (i \rho) = \frac{(2d)^d}{(d!)^2} \prod_{i=1}^d \sin^2 (i \rho) ,
	\end{equation}
	which gives the claim. 
\end{proof}

Formula \eqref{eq:detVdCheby} simplifies for $ \rho = \pi/(2d) $. The relevance of this case, which lies at the intersection of classes (C1) and (C2), is described in \cite{BEsinum}. This value will play an important role also in the next sections.

\begin{corollary}
    If $ \rho = \pi/(2d) $, then \eqref{eq:detVdCheby} reduces to
	$$ \left| \det V_d \right| = 
    \sqrt{\frac{d^{d+1}}{(d!)^2 2^{d-3}}}.$$
\end{corollary}

  \begin{proof}
    
    To prove this fact, we shall invoke the trigonometric equality
	\begin{equation} \label{eq:trigonometricequality}
		\prod_{i=1}^{2d-1} \sin \left( \frac{i\pi}{2d} \right) = \frac{2d}{2^{2d-1}} .
	\end{equation}
	Notice, in particular, that we can write
	$$ \prod_{i=1}^{2d-1} \sin \left( \frac{i\pi}{2d} \right) = \prod_{i=1}^{d-1} \sin \left( \frac{i\pi}{2d} \right) \cdot \sin{ \left(\frac{d \pi}{2d} \right)} \cdot \prod_{i=d+1}^{2d-1} \sin \left( \frac{i\pi}{2d} \right) = \prod_{i=1}^{d-1} \sin \left( \frac{i\pi}{2d} \right) \cdot \prod_{i=d+1}^{2d-1} \sin \left( \frac{i\pi}{2d} \right) .$$
	By the symmetry of the sine function $ \sin (\pi - t) = \sin (t) $, by the change of indices $ j = 2d - i $, one immediately retrieves that
	$$ \prod_{i=d+1}^{2d-1} \sin \left( \frac{i\pi}{2d} \right) = \prod_{i=1}^{d-1} \sin \left( \frac{i\pi}{2d} \right) = \prod_{i=1}^{d} \sin \left( \frac{i\pi}{2d} \right) ,$$
	where the last equality follows as the $ d$-th term of the product is $ \sin( \pi/2) =  1 $. Plugging this into \eqref{eq:trigonometricequality}, one obtains
	$$ \prod_{i=1}^{d-1} \sin^2 \left( \frac{i\pi}{2d} \right) = \frac{2d}{2^{2d-1}} = \frac{d}{2^{2(d-1)}} .$$
	Plugging this into \eqref{eq:partialprod}, one finds
	$$ \prod_{i=1}^d \frac{2d}{i^2} \sin^2 (i \rho) = \frac{(2d)^d}{(d!)^2} \frac{d}{2^{2(d-1)}} = \frac{d (2d)^d}{(d! 2^{d-1})^2} ,$$
	which immediately yields the claim by including the factor $ 2 $, simplifying and taking the square root.
\end{proof}

We close this section with an interesting fact. The identity \eqref{eq:trigonometricequality} is known, but the authors struggled to find a reference.  
For completeness, we report a possible proof.

\begin{lemma}
	The trigonometric equality \eqref{eq:trigonometricequality} holds; that is, for any $ n \in \N $, one has
	$$ \prod_{i=1}^{n-1} \sin \left( \frac{i\pi}{n} \right) = \frac{n}{2^{n-1}} . $$
\end{lemma}

\begin{proof}
	By multiplying both sides for $ 2^{n-1} $ and using Euler's formula, we may rewrite left hand side of the identity as
	\begin{align*}
		\prod_{j=1}^{n-1} 2 \sin \left( \frac{j\pi}{n} \right) & = \prod_{j=1}^{n-1} 2 \sin \left( \frac{j\pi}{n} \right) = \prod_{j=1}^{n-1} \frac{e^{i \frac{j \pi}{n}} - e^{i \frac{j \pi}{n}} }{i} = (-i)^{n-1} \prod_{j=1}^{n-1} e^{-i \frac{j \pi}{n}} \left( e^{2 i \frac{j \pi}{n}} - 1 \right) \\
		& = (-i)^{n-1} \prod_{j=1}^{n-1} e^{-i \frac{j \pi}{n}} \prod_{j=1}^{n-1} \left( e^{2 i \frac{j \pi}{n}} - 1 \right) = (-i)^{n-1} e^{-i \sum_{j=1}^{n-1} \frac{j \pi}{n}} \prod_{j=1}^{n-1} \left( e^{2 i \frac{j \pi}{n}} - 1 \right) \\
		& = (-i)^{n-1} e^{-i (n-1) \pi} \prod_{j=1}^{n-1} \left( e^{2 i \frac{j \pi}{n}} - 1 \right) = \prod_{j=1}^{n-1} \left( 1 - e^{2 i \frac{j \pi}{n}} \right) .
	\end{align*}
	We shall prove that this quantity now equals to $ n $. The last term can be treated analyzing the roots of the unity, as
	$$ z^{n} - 1 = \prod_{j=0}^{n-1} \left( z - e^{2i \frac{j \pi}{n}} \right) = (z-1) \prod_{j=1}^{n-1} \left( z - e^{2i \frac{j \pi}{n}} \right), $$
	whence, by dividing both sides by $ z-1 $ and computing the quotient, it follows that
	$$ \prod_{j=1}^{n-1} \left( z - e^{2 i \frac{j \pi}{n}} \right) = \sum_{j=0}^{n-1} z^j .$$
	Evaluating the expression at $ z = 1 $ we get the claim.
\end{proof}

\subsection{An inversion formula for the Chebyshev segments} \label{sect:inverseC2}

The orthogonality of the Vandermonde matrix can be used to determine its inverse, and in particular the Lagrange basis associated with segments in the class (C2).

\begin{remark}
	The columns of $ V_d ^{-1} $ are the coefficients that turn a given basis into the Lagrange one, i.e. the only (up to permutations) basis $ \{ \ell_1, \ldots, \ell_d \} $ for $ \mathscr{P}_{d-1} (I) $ such that $ \int_{s_i} \ell_j (x) \de x = \delta_{i,j} $, see \cite[Proposition 1]{BSSC}. An explicit formula for the inverse of $ V_d $ is crucial, as such bases in histopolation are known only for segments in the classes (C1) and (C3).
\end{remark}

For ease of the next computations, let us set once more $ W_d := V_d^\top V_d $. In Proposition \ref{prop:orthogonality} we proved that $ W_d $ is a diagonal invertible matrix, whenever the segments are unisolvent for histopolation. Hence, by left-multiplying for $ W_d^{-1} $ and right-multiplying by $ V_d^{-1} $, one gets
\begin{equation} \label{eq:inverseVandermonde}
	V_d^{-1} = W_d^{-1} V_d^\top .
\end{equation}
The elements for $ V_d^\top $ are given in \eqref{eq:vandermondeC2}, while those of $ W_d^{-1} $ are immediately retrieved by \eqref{eq:diagonalterm}. Now, in general one has:
\begin{align*}
	\left[ V_d^{-1}\right]_{i,j} & = \left[ W_d^{-1} V_d^\top  \right]_{i,j} = \sum_{k=1}^d \left[ W_d^{-1} \right]_{i,k} \left[ V_d^\top \right]_{k,j} = \sum_{k=1}^d \left[ W_d^{-1} \right]_{i,k} \left[ V_d \right]_{j,k} .
\end{align*}
Since $ W_d^{-1} $ is diagonal, plugging in the explicit expressions of the terms computed in \eqref{eq:diagonalterm}, one has
\begin{align*}
	\left[ V_d^{-1} \right]_{i,j} = \left[W_d^{-1} \right]_{i,i} \left[ V_d \right]_{i,j} = \frac{i^2}{2d \sin^2 (i \rho)} \cdot \frac{2 \sin (i \tau_j) \sin (i \rho)}{i} = \frac{i \sin (i \tau_j)}{d \sin (i \rho)} 
\end{align*}
for $ i \ne d $, and 
\begin{align*}
	\left[ V_d^{-1} \right]_{d,j} = \left[W_d^{-1} \right]_{d,d} \left[ V_d \right]_{d,j} = \frac{d^2}{4d \sin^2 (d \rho)} \cdot \frac{2 \sin (d \tau_j) \sin (d \rho)}{d} = \frac{\sin (d \tau_j)}{2 \sin (d \rho)} 
\end{align*}
for $ i = d $ (i.e., the same formula as above, with the corrective factor $ 2 $).

\subsection{A formula for the conditioning of Chebyshev segments} \label{sect:condCheb}

An immediate consequence of Proposition \ref{prop:orthogonality} is that the singular values of $ V_d $ are the diagonal entries of $ V_d^\top V_d $. A direct inspection on \eqref{eq:diagonalterm} shows that the ordering of the diagonal elements of $ V_d^\top V_d $ depends on the relationship between $ \rho $ and $ d $. Nevertheless, monotonicity conditions can be explicitly stated.

\begin{lemma} \label{lem:decreasing}
	Let $ 0 < \rho < \frac{\pi}{d} $. Then for $ 1 \leq i < j \leq d-1 $, one has
	$$ \left[ V_d^\top V_d \right]_{i,i} = \frac{2d}{i^2} \sin^2 (i \rho) > \frac{2d}{j^2} \sin^2 (j \rho) = \left[ V_d^\top V_d \right]_{j,j} . $$
\end{lemma}

\begin{proof}
	Consider the function $ f(x) = \sin^2 (x \rho) / x^2 $, for $ x > 0 $. One computes that
	$$ f'(x) = \frac{2 \sin (\rho x) \left(\rho x \cos (\rho x) - \sin (\rho x) \right) }{x^3} ,$$
	whence $ f'(x) < 0 $ for $ 0 < x < \pi/\rho $. This implies that, for $ i,j = 1, \ldots, d-1 $ and $ i < j $,
	$$ \frac{ \sin^2 (i \rho)}{i^2} > \frac{ \sin^2 (j \rho)}{j^2} ,$$
	which immediately yields the claim.
\end{proof}

\begin{remark} \label{rmk:pimezzi}
	Although the connection between $ \rho $ and $ d $ may seem a technical condition, such an inverse relationship between $ \rho $ and $ d $ is in fact quite natural, as discussed in \cite{BEsinum}. 
	Indeed, taking $ \rho = \frac{\pi}{2d} $, arcs overlap in their endpoints, and one obtains segments that belong to both classes (C1) and (C2). Any smaller $ \rho > 0 $ gives non-overlapping segments, while $ \rho = \frac{\pi}{d} $ results in a loss of unisolvence, as evident from \eqref{eq:vandermondeC2}. From now on we hence assume that the arc-length parameter is of the form $ \rho = a \frac{\pi}{d} $, with $ 0 < a < 1 $. This choice is consistent with the literature \cite{BEfekete} and ensures the well-posedness of the histopolation problem \cite{BEsinum}.
\end{remark}

Lemma \ref{lem:decreasing} states that the all diagonal terms of $ V_d^\top V_d $, except the last one, are decreasingly ordered, provided that $ \rho $ is sufficiently small. 
%
Since $ V_d^\top V_d $ is a diagonal matrix, the ratio between its largest and smallest element gives the Euclidean conditioning of $ V_d $. Hence, to obtain an explicit formula for $ \kappa_2 (V_d) $, we only need to understand how the last diagonal term of $W_d := V_d^\top V_d $ relates with the others. 

\begin{lemma} \label{lem:doublebound}
	Let $ \frac{\pi}{2d} \leq \rho < \frac{\pi}{d} $ (i.e., $ 1/2 \leq a < 1 $). One has
	\begin{equation} \label{eq:sigma1}
		[W_d]_{1,1} = 2d \sin^2 (\rho) > \frac{4}{d} \sin^2 (d \rho) = [W_d]_{d,d}
	\end{equation}
	Let $ 0 < \rho \leq \frac{\pi}{2d} $ (i.e., $ 0 < a \leq 1/2 $). One has
	\begin{equation} \label{eq:sigmad}
		[W_d]_{d,d} =  \frac{4}{d} \sin^2 (d \rho) >  \frac{2d}{(d-1)^2} \sin^2 ((d-1) \rho) = [W_d]_{d-1,d-1} .
	\end{equation}
	Further, for any $ 0 < \rho < \frac{\pi}{d} $ (i.e., for any $ 0 < a < 1 $), one has 
	\begin{equation} \label{eq:limitratio}
		\lim_{d \to \infty} \frac{[W_d]_{d,d}}{[W_{d}]_{d-1,d-1}} = 2 .
	\end{equation}
\end{lemma}

\begin{proof}
	The claims for \eqref{eq:sigma1} and \eqref{eq:sigmad} can be obtained by an analogous analysis to that of Lemma \ref{lem:decreasing}. To obtain \eqref{eq:limitratio}, we plug the explicit characterizations given in \eqref{eq:diagVtV} and compute
	\begin{align*}
		\lim_{d \to \infty} \frac{[W_d]_{d,d}}{[W_{d}]_{d-1,d-1}} & = \lim_{d \to \infty} \frac{\frac{4d}{d^2} \sin^2 (\rho d)}{\frac{2d}{(d-1)^2} \sin^2 ((d-1) \rho)} = \lim_{d \to \infty} \frac{2 (d-1)^2  \sin^2 (a \frac{\pi}{d} d)}{d^2 \sin^2 ((d-1) a \frac{\pi}{d})}  \\
		& = 2 \lim_{d \to \infty} \frac{(d-1)^2}{d^2} \frac{\sin^2 ( a \pi ) }{ \sin^2 (a\frac{d-1}{d} \pi)} = 2.
	\end{align*}
	Observe that the above computation is meaningful, as $ 0 < a\frac{d-1}{d} \pi < a \pi < \pi $, so the denominator never vanishes.
\end{proof}

Lemma \ref{lem:doublebound} has two important consequences. First, it evidences the limiting role of the case $ \rho = \pi/(2d) $ even beyond the context of interpolation discussed in Remark \ref{rmk:pimezzi}. 
By a continuity argument, for arc-lengths around such a value, merging \eqref{eq:sigma1} and \eqref{eq:sigmad} one deduces the simple formula
\begin{align} \label{eq:simplifiedcond}
	\kappa_2 \left( V_d \right) & = \frac{\sigma_1 (V_d)}{\sigma_d (V_d)} = \sqrt{\frac{\lambda_1 (V_d^\top V_d )}{\lambda_d (V_d^\top V_d )}} = \sqrt{\frac{[V_d^\top V_d ]_{1,1}}{[V_d^\top V_d ]_{d-1,d-1}} } = \sqrt{\frac{(d-1)^2 \sin^2(\rho)}{\sin^2 ((d-1) \rho)}} \\
	& = \frac{(d-1) \sin (\rho)}{\sin ((d-1) \rho)} \nonumber
\end{align}
for the conditioning of $ V_d $. Note the abosence of the absolute value in the last term, which has been removed due to the positivity of the sine function for $ 0 < \rho = a \pi /d < a (d-1)/d \pi = (d-1) \rho < \pi $ with $ 0 < a < 1 $. 
Further, as $ \kappa_2 (V_d) $ can be obviously lower bounded by the ratio of any two diagonal terms of $ W_d $, for large values of $ d $ we may take the $(d-1)$-th and the $ d $-th to obtain
$$ \kappa_2 (V_d) \geq \sqrt{\frac{[V_d^\top V_d ]_{d,d}}{[V_d^\top V_d ]_{d-1,d-1}} } ,$$
and owing to \eqref{eq:limitratio} one also gets the asymptotic lower bound
\begin{equation} \label{eq:asymptoticlowerbound}
	\lim_{d \to \infty} \kappa_2 (V_d) \geq \sqrt 2 .
\end{equation}

\begin{theorem} \label{thm:conditioningC2}
For the Chebyshev segments belonging to the intersection of classes (C1) and (C2), i.e. for $ \rho = a \frac{\pi}{d}$ with $ a = 1/2 $, the conditioning $ \kappa_2 (V_d) $ is monotonically increasing towards
	$$ \lim_{d\to \infty}\kappa_2 \left( V_d \right) = \frac{\pi}{2} . $$
\end{theorem}

\begin{proof}
	Plugging the condition $ \rho = \frac{\pi}{2d} $ into \eqref{eq:simplifiedcond}, one gets that
	$$ \kappa_2 \left(V_d \right) = \frac{(d-1) \sin (\rho)}{\sin ((d-1) \rho)} = \frac{(d-1) \sin (\frac{\pi}{2d} )}{\sin ((d-1) \frac{\pi}{2d} )} .$$
	By considering the limit $ d \to \infty $, and setting $ t = \pi/(2d) $, one has
	$$ \lim_{d \to \infty} \kappa_2 \left( V_d \right) = \lim_{d \to \infty} \frac{(d-1) \sin (\frac{\pi}{2d} )}{\sin ((d-1) \frac{\pi}{2d} )} = \frac{1}{\sin (\frac{\pi}{2}) }\lim_{t \to 0} \frac{\left( \pi - 2t \right) \sin (t)}{2t} = \frac{\pi}{2} .$$
	To show the monotonicity of the conditioning, we shall prove that
	$$ \frac{d \sin \left(\frac{\pi}{2d}\right)}{\sin \left( d \frac{\pi}{2 d }\right)} = \frac{ d \sin \left(\frac{\pi}{2d}\right)}{\sin \left( \frac{\pi}{2}\right)} \geq \frac{ (d-1 )\sin \left(\frac{\pi}{2(d-1)}\right)}{\sin \left( \frac{\pi}{2}\right)} = \frac{(d-1) \sin \left(\frac{\pi}{2(d-1)} \right)}{\sin \left((d-1)\frac{\pi}{2 (d-1)} \right)} $$
	for each $ d \geq 1 $. This can be done by showing that the function $ f(x) := x \sin (\pi/(2x)) $ is increasing for $ x > 1 $, i.e., that $ f'(x) > 0 $ for $ x > 1 $. After differentiation, making the change of variable $ t = \pi/(2x) $, it is readily seen that this is equivalent to proving that $ g(t) := \tan (t) - t > 0 $ for $ t \in (0, \pi/2) $. The first derivative of this function is $ g'(t) = \tan^2 (t) $, which is clearly defined and positive in such an interval, and since $ \lim_{t \to 0^+} g(t) = 0^+ $, this concludes the proof.
\end{proof}

A numerical verification of the case $ a = 1/2 $ of Theorem \ref{thm:conditioningC2} is provided in Figure \ref{fig:condexampleC2}.

\begin{figure}[h!]
	\centering
	\includegraphics[width=0.75\textwidth]{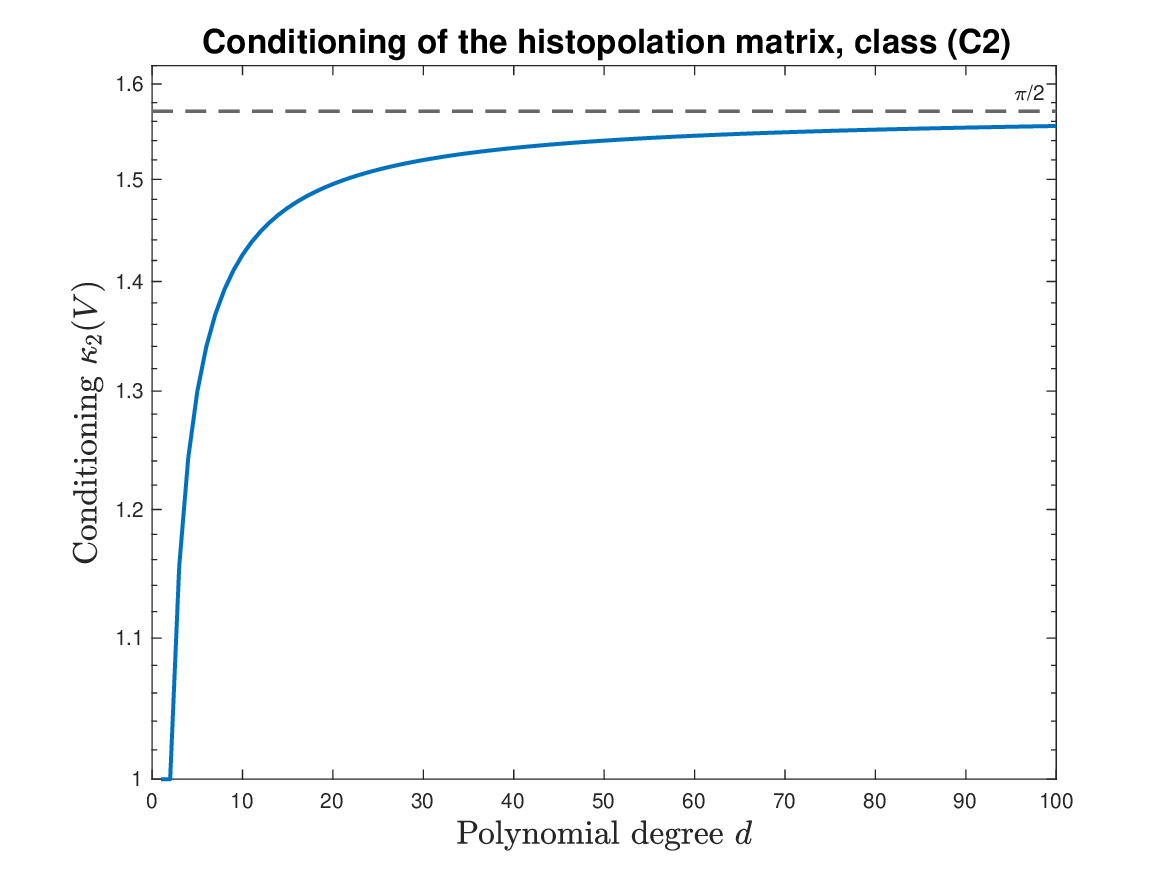}
	\caption{Bounded conditioning of the Vandermonde matrix in the Chebyshev basis for the class (C2). The parameter is $ a = 1 / 2 $, so $ \rho = \pi/(2d) $. The resulting segments also belong to the class (C1).} \label{fig:condexampleC2}
\end{figure}


\subsubsection{Asymptotics for $ a $ far from $ 1/2 $}

Lemma \ref{lem:doublebound}, together with an immediate continuity argument, showed that the conditioning of $V_d$ around $ a = 1/2 $ is described by the square root of the ratio between $ [W_d]_{1,1} $ and $ [W_d]_{d-1,d-1} $. For small perturbations, hence, Theorem \ref{thm:conditioningC2} provides a good approximation of the limit value. However, if one considers $ a $ either close to $ 0 $ or $ 1 $, it is possible to reverse the inequalities \eqref{eq:sigma1} and \eqref{eq:sigmad}.

\begin{example} \label{ex:nobound}
	Consider $ a = 1/10 $ and $ d = 3 $. One has
	$$ [W_3]_{3,3} = \frac{4}{3} \sin^2 \left( \frac{\pi}{10} \right) \approx 0.1274 > 0.0656 \approx 6 \sin^2 \left( \frac{\pi}{30} \right) = [W_3]_{1,1} .$$
	Consider $ a = 9/10 $ and $ d = 3 $. One has
	$$ [W_3]_{3,3} = \frac{4}{3} \sin^2 \left( \frac{9}{10} \pi \right) \approx 0.1274 < 1.8090 \approx 2 \sin^2 \left( \frac{3}{5} \pi \right) = [W_3]_{2,2} .$$
\end{example}

From the perspective of an asymptotic analysis, Example \ref{ex:nobound} does not constitute a real obstacle, as we are interested in the relationship between the elements of $ W_d $ for large values of $ d $. To this end, we gather the results of Section \ref{sect:condCheb} in the following Proposition, which is an immediate consequence of Lemma \ref{lem:decreasing} and Lemma \ref{lem:doublebound}.

%

\begin{proposition} \label{prop:tresholding}
	Let $ 0 < a < 1 $. One has
	\begin{equation} \label{eq:genericconditioning}
		\kappa_2^2 (V_d) = 
			\begin{cases}
				\frac{[W_{d}]_{1,1}}{[W_d]_{d,d}} & \text{if } [W_d]_{d,d} < [W_d]_{d-1,d-1}, \\
				\frac{[W_{d}]_{d,d}}{[W_d]_{d-1,d-1}} & \text{if } [W_d]_{d,d} \geq [W_d]_{1,1}, \\
				\frac{[W_{d}]_{1,1}}{[W_d]_{d-1,d-1}} & \text{if } [W_d]_{1,1} \geq [W_d]_{d,d} \geq [W_d]_{d-1,d-1} .
			\end{cases}
	\end{equation}
\end{proposition}

It remains to establish which case one falls in. First of all, recall \eqref{eq:limitratio} from Lemma \ref{lem:doublebound}. Such a formula implies that, for large values of $ d $, the first case of \eqref{eq:genericconditioning} is empty. We now show that, for large values of $ d $, there exists a thresholding value $ a^* $ for $ 0 < a < 1 $ which separates the remaining two situations in Proposition \ref{prop:tresholding}.

\begin{proposition}
	There exists $ a^* $, with $ 0 < a^* < 1 $, such that for each $ 0 < a < a^* $ one has
	\begin{equation} \label{eq:a*}
		\lim_{d \to \infty} \frac{[W_d]_{d,d}}{[W_d]_{1,1}} > 1 .
	\end{equation}
\end{proposition}

\begin{proof}
	To establish \eqref{eq:a*} one computes, applying the substitution $ t = a \pi / d $,
	\begin{align*} 
		\lim_{d \to \infty} \frac{[W_d]_{d,d}}{[W_d]_{1,1}} & = \lim_{d \to \infty} \frac{\frac{4}{d} \sin^2 ( d a \frac{\pi}{d})}{2d \sin^2 ( a \frac{\pi}{d})} = \lim_{d \to \infty} \frac{\sin^2 (a \pi)}{\frac{d^2}{2} \sin^2 (a \frac{\pi}{d})} = 2 \sin^2 (a \pi) \lim_{d \to \infty} \frac{1}{d^2 \sin^2 \left( a \frac{\pi}{d} \right)} \\
		& =  2 \sin^2 (a \pi) \lim_{t \to 0} \frac{t^2 (t)}{\sin^2(t)} \frac{1}{a^2 \pi^2} = \frac{2 \sin^2 (a \pi) }{a^2 \pi^2} .
	\end{align*}
	By computing an approximate solution of $ \frac{2 \sin^2 (a \pi) }{a^2 \pi^2} > 1 $, one obtains $ a^* \approx 0.40 $.
\end{proof}

For large $ d $, one may hence make the claim of Proposition \ref{prop:tresholding} more precise.

\begin{theorem} \label{thm:asymptcond}
	One has
	\begin{equation} \label{eq:asymptoticconditioning}
		\lim_{d \to \infty} \kappa_2 (V_d) = \begin{cases}
			\frac{a \pi}{\sin (a \pi)} & \text{if } a^* < a < 1, \\
			\sqrt{2} & \text{if } 0 < a < a^* .
			\end{cases}
	\end{equation}
\end{theorem}

\begin{proof}
	Suppose first that $ a \in (0, a^*) $. Then, by \eqref{eq:a*}, for large $ d $ one has $ [W_d]_{d,d} > [W_d]_{1,1} $. Due to Lemma \ref{lem:decreasing}, one in general has $ [W_d]_{1, 1} > [W_d]_{i,i} $ for any $ i = 2, \ldots, d-1 $, hence
	$$ \lim_{d \to \infty} \kappa_2 (V_d) =  \lim_{d \to \infty} \sqrt{\frac{[W_d]_{d,d}}{[W_{d}]_{d-1,d-1}}} =  \sqrt{2} ,$$
	as computed in \eqref{eq:limitratio}.
	
	Suppose then that $ a \in (a^*, 1) $. By \eqref{eq:a*}, for large $ d $ one has that in such an interval $ [W_d]_{1,1} > [W_d]_{d,d} $. Further, by \eqref{eq:limitratio}, we deduce that in this limit $ [W_d]_{d,d} > [W_{d}]_{d-1,d-1} $. It follows that
	$$ \lim_{d \to \infty} \kappa_2 (V_d) =  \lim_{d \to \infty} \sqrt{\frac{[W_d]_{1,1}}{[W_{d}]_{d-1,d-1}}} =  \lim_{d \to \infty}
	\sqrt{ \frac{2d \sin^2 \left( a \frac{\pi}{d}\right)}{\frac{2d}{(d-1)^2} \sin^2 \left( a \pi \frac{(d-1)}{d}\right)}} $$
	and, by applying the same substitution $ t = a \pi /d $ and reasoning of Proposition \ref{prop:tresholding}, one readily computes that
	$$ \lim_{d \to \infty} \kappa_2 (V_d) = \frac{a \pi}{\sin(a \pi)} .$$
	This concludes the proof.
\end{proof}

Note that \eqref{eq:asymptoticconditioning} has two interesting implications: first, it allows to extend, in an asymptotic perspective, the formula \eqref{eq:simplifiedcond} to a large family of segments in the class (C2). Further, it shows that for $ a < a^* $ the conditioning plateaus around the optimal value, in the sense of the lower bound \eqref{eq:asymptoticlowerbound}. A numerical verification of Theorem \ref{thm:asymptcond} is given in Figure \ref{fig:asymptotic}.

\begin{figure}[h!]
	\centering
	\includegraphics[width=0.49\textwidth]{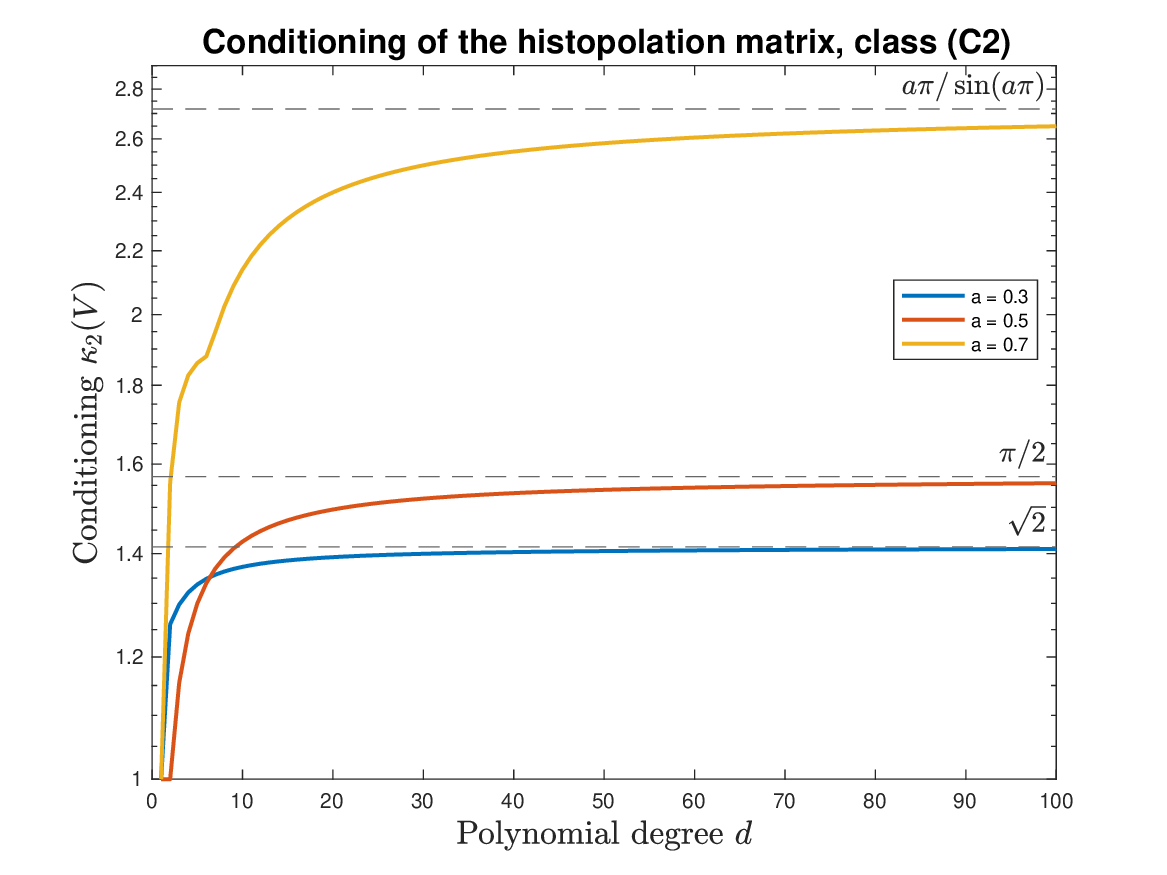}
	\includegraphics[width=0.49\textwidth]{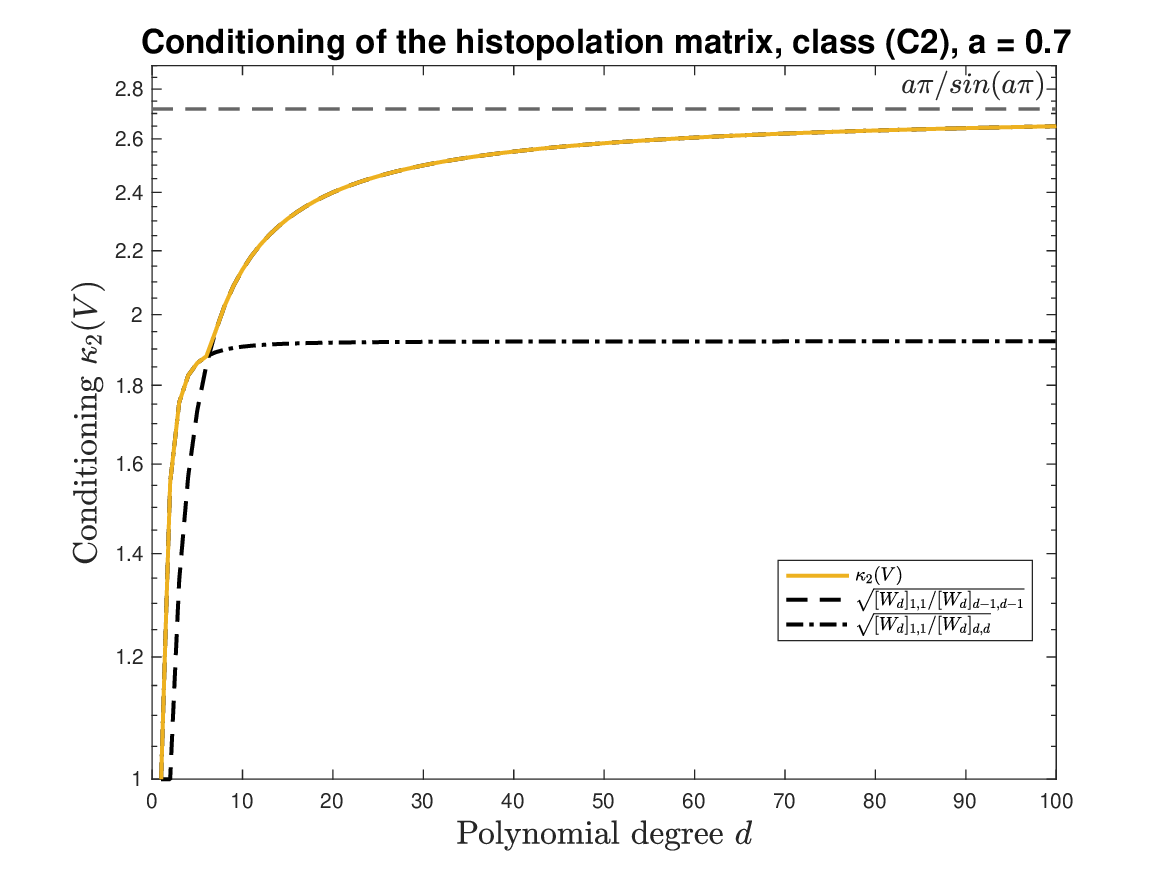}
	\caption{Bounded conditioning of the Vandermonde matrix in the Chebyshev basis for the class (C2), for different values of $ a $. Observe that, for $ a = 0.7$ (orange line in both panels), for $ d = 6 $ there is an exchange in the relationship between $ [W_{d}]_{1,1} $ and $ [W_d]_{d,d} $.} \label{fig:asymptotic}
\end{figure}

\begin{remark}
	The situation pictured in Theorem \ref{thm:asymptcond} is asymptotic. In practice, for intermediate values of $ d $, the first case of Proposition \ref{prop:tresholding} may happen. 
	Likewise, it may happen that, if $ a $ is fixed, there exists a finite value $ d^* $ for which the largest (respectively the smallest) eigenvalue of $ W_{d^*} $ exchanges its role with $ [W_{d^*}]_{d^*, d^*} $. This value $ d^* $ gives the peculiar step pictured in the line corresponding to $ a = 0.7 $ in Figure \ref{fig:asymptotic}. A detail of this phenomenon is pictured at the right hand panel of the same Figure.

\end{remark}

\subsubsection{The effect of the Chebyshev basis on generic segments}

Let us now remark an interesting aspect. The selection of the basis is crucial to avoid an exponential conditioning of the Vandermonde matrix. Nevertheless, without an appropriate selection of histopolation segments, Chebyshev polynomials are not sufficient to control the conditioning of the Vandermonde matrix. This is depicted in the left hand panel of Figure \ref{fig:wrongbases}, where the conditioning of the Chebyshev basis (blue solid line) on equidisitributed segments of class (C1) is shown. Although there is a clear improvement with respect to the monomial basis (red dashed line), the exponential growth for the conditioning is not avoided. On the other hand, the choice of the family of segments is of interest as well: the right hand panel of Figure \ref{fig:wrongbases} shows the conditioning of the monomial basis (blue solid line) for segments in the class (C2), which is smaller than that obtained for segments in class (C1) for the corresponding polynomial degrees (again, red dashed line). For comparison, equidistributed segments in the class (C1) on $ I = [-1,1] $ were considered.

\begin{figure}[h!]
	\includegraphics[width=0.49\textwidth]{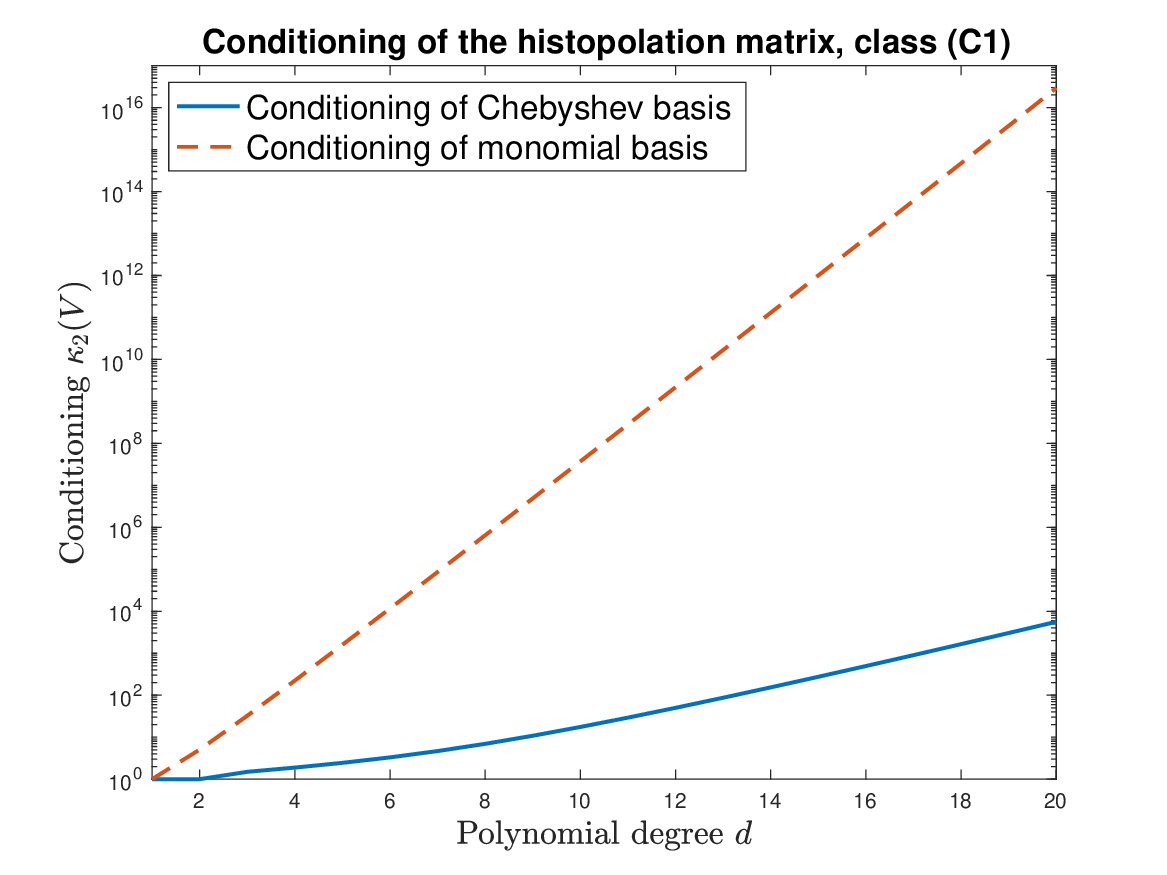}
	\includegraphics[width=0.49\textwidth]{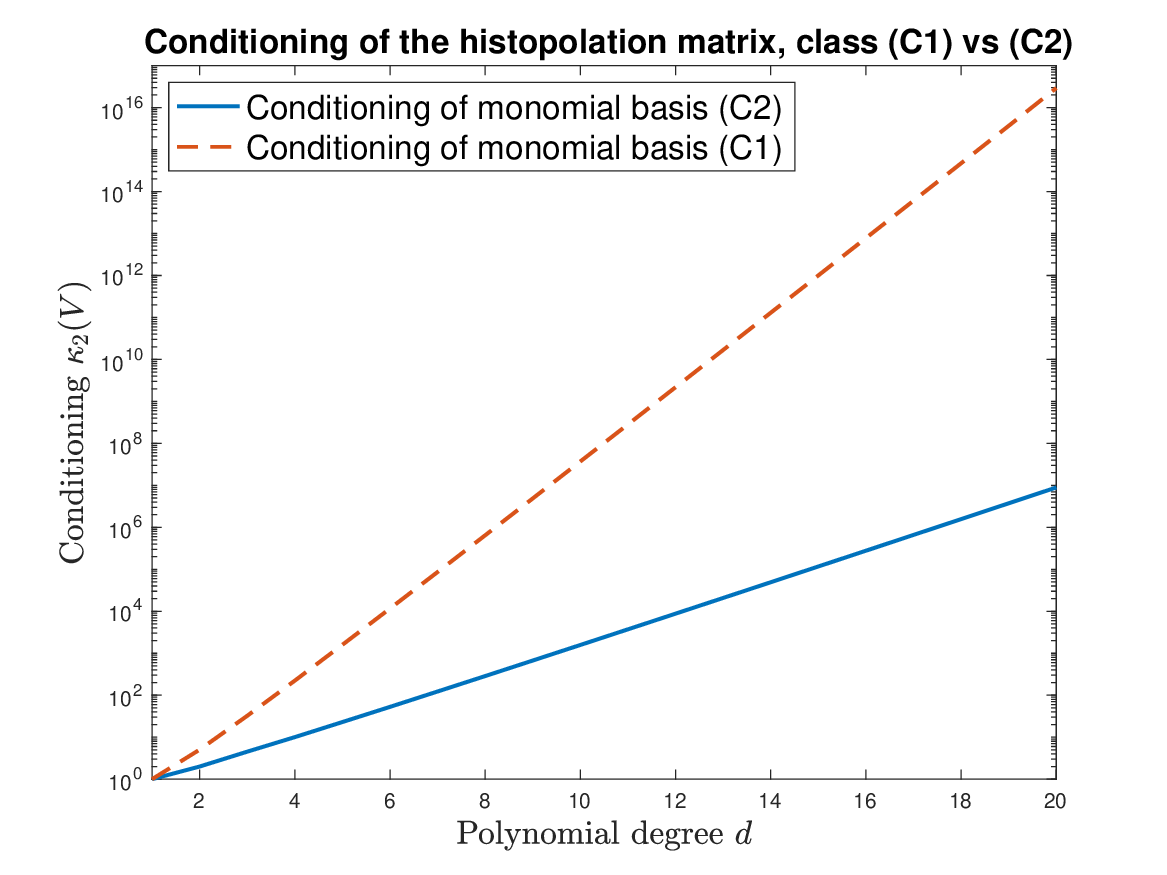}
	\caption{Left: conditioning of the Chebyshev basis for segments in the class (C1). Right: conditioning of the monomial basis for segments in class (C2), compared with segments in class (C1).} \label{fig:wrongbases}
\end{figure}

\subsubsection{The normalized case}

For the class (C2), the normalization factor $|s_i| = | \cos(\tau_i - \rho) - \cos(\tau_i + \rho)|  = 2 | \sin (\tau_i) \sin (\rho) | $ turns out to be useful in the computation of the Lebesgue constants performed in \cite{BEsinum}. Supposing that $ \tau_i $ and $ \rho $ satisfy the hypotheses of the above section, so that $ \sin (\tau_i ) > 0 $ and $ \sin (\rho) > 0 $, the corresponding Vandermonde matrix is
\begin{equation} \label{eq:vandermondeC2norm}
	\left[\widetilde V_d \right]_{i,j} = \frac{\frac{2}{j} \sin (j \tau_i) \sin (j \rho)}{2 \sin (\tau_i) \sin (\rho)} = \frac{\sin (j \tau_i) \sin (j \rho)}{j \sin (\tau_i) \sin (\rho)}.
\end{equation}
Let us observe that $ \widetilde{V}_d = S_d V_d $, where $ S_d $ is the diagonal matrix
    $$ [S_d]_{i,i} = \frac{1}{2 \sin (\tau_i) \sin (\rho)} $$
    and $ V_d $ is the non-normalized matrix \eqref{eq:VandermondeU}. Using Binet's Theorem and Corollary \ref{cor:detC2}, we may hence provide a closed formula for $ \det \widetilde V_d $.

\begin{proposition} \label{prop:detVtilded}
    Let $ 0 < \rho < \pi/d $. One has
    $$ \det \widetilde V_d = \frac{\det V_d}{2 \sin^d (\rho)}, $$
    where a formula for the modulus of $ \det V_d $ is given in \eqref{eq:detVdCheby}, and the sign is $ (-1)^{\frac{d(d-1)}{2}} $.
\end{proposition}

\begin{proof}
    It is sufficient to compute $ \det S_d = \left( \prod_{i=1}^d 2 \sin (\tau_i) \sin (\rho) \right)^{-1} $. For, we compute
    \begin{align*}
        \prod_{i=1}^d 2 \sin (\tau_i) \sin (\rho) = \left(\prod_{i=1}^d 2\right) \left(\prod_{i=1}^d \sin (\tau_i) \right) \left(\prod_{i=1}^d \sin (\rho) \right) = 2^d \sin^d (\rho)  \left(\prod_{i=1}^d \sin (\tau_i) \right) .
    \end{align*}
    Applying the product formula $ \sin (d x) = 2^{d-1} \prod_{i=0}^{d-1} \sin \left( \pi i/d + x \right) $ to Chebyshev nodes, and shifting indices, one readily deduces that $\prod_{i=1}^d \sin (\tau_i) = 2^{1-d} $, whence
    $$ \det S_d = \frac{1}{2^d 2^{1-d} \sin^d (\rho)} = \frac{1}{2 \sin^d (\rho)} .$$
    Binet's Theorem immediately implies the claim.
\end{proof}

When $ \rho = \pi/(2d) $, one may provide an asymptotics for $ \det \widetilde V_d $. It follows immediately from Corollary \ref{cor:detC2} and the limit $ \lim_{d \to \infty} \sin^d (\pi/(2d)) = (\pi/(2d))^d $.

\begin{corollary}
    Let $ \rho = \pi/(2d) $. One has
    $$ \lim_{d \to \infty} \left| \det \widetilde V_d \right| = \frac{\sqrt{d^{3d+1} 2^{d-4}}}{d! \pi^d} .$$
\end{corollary}
%

Of course, due to the additional left-multiplication by $ S_d $, $ \widetilde V_d $ does not satisfy anymore the orthogonality condition of Proposition \ref{prop:orthogonality}. Nevertheless, the conditioning $ \kappa_2 (\widetilde V_d) $ does not grow faster than linearly.

\begin{proposition} \label{prop:normCheby}
    For every $ \rho \in (0, \pi/d) $, one has
    $$ \lim \kappa_2 (\widetilde V_d) \leq \frac{2 \kappa_2 (V_d)}{\pi} d .$$
\end{proposition}

\begin{proof}
    Let us write $ \widetilde V_d = S_d V_d $ and analyze $ \kappa_2 (S_d) $. Clearly, factorizing the term $ 2 \sin (\rho) $, one finds
    $$ \kappa_2 (S_d) = \frac{\max_{i} \sin (\tau_i)}{\min_i \sin (\tau_i)} \leq \frac{1}{\min_i \sin (\tau_i)} .$$
    The minimum of $ \sin (\tau_i) $ is reached for $ \tau_1 = \pi/(2d) $, the closest to zero Chebyshev node, and for $ \tau_d = (2d-1)\pi/(2d) $. Considering $ \tau_1 $, as $ d \to \infty$, one has $ \sin (\pi/2d) \approx \pi / 2d $, which implies that 
    $$ \kappa_2 (S_d) \leq \frac{2d}{\pi} .$$
    Since $ \kappa_2 (\widetilde{V}_d) \leq \kappa_2 (S_d) \kappa_2 (V_d) $, the result immediately follows from Theorem \ref{thm:asymptcond}.
\end{proof}

Notice that, for $ \rho = \frac{\pi}{2d} $, which corresponds to the parameter $ a = 1/2 $, the upper bound established by Proposition \ref{prop:normCheby} simply becomes $ \kappa_2 (\widetilde V_d) \leq d $. The behavior of the conditioning is confirmed by the numerical trend depicted in Figure \ref{fig:condexampleC2Norm}.

\begin{figure}[h!]
	\centering
	\includegraphics[width=0.75\textwidth]{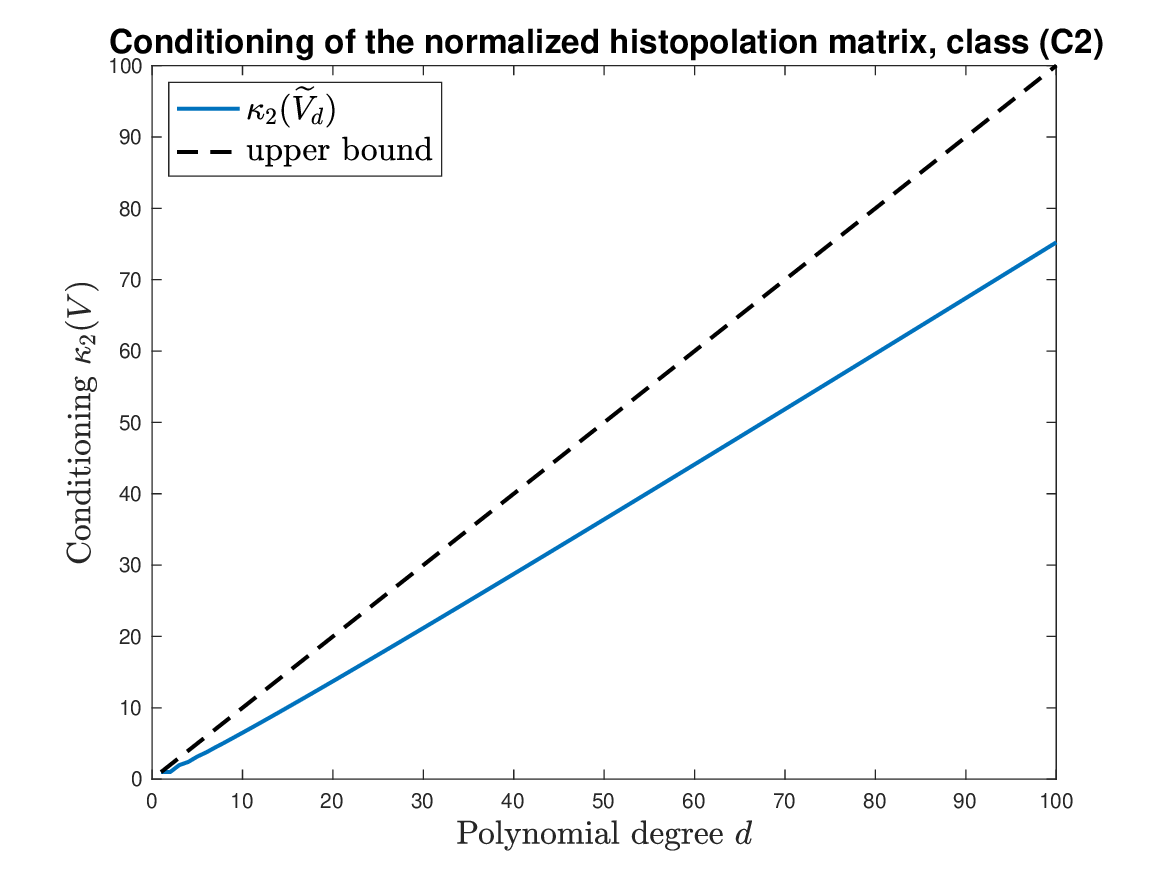}
	\caption{Linear conditioning of the Vandermonde matrix in the Chebyshev basis for the class (C2). The parameter is $ a = 1 / 2 $, so $ \kappa_2 (\widetilde V_d) \leq d $. The resulting segments also belong to the class (C1).} \label{fig:condexampleC2Norm}
\end{figure}

\subsection{Conditioning with respect to the Frobenius norm}

Let us recall that, for a matrix $ A \in \R^{m \times n}$, its \emph{Forbenius norm} is the quantity
$$ \Vert A \Vert_{F} := \sqrt{\sum_{i=1}^m \sum_{i=1}^n | A_{i,j} |^2 } = \sqrt{\mathrm{Tr} (A^\top A)} $$
and the corresponding condition number is thus
\begin{equation} \label{eq:Frobcond}
	\kappa_F (A) := \Vert A \Vert_F \Vert A^{-1} \Vert_F .
\end{equation}
The quantity $ \kappa_F (A) $ reflects some geometrical features of the matrix $ A $ \cite{Chehab}, and in the context of nodal interpolation with Chebyshev nodes it has been studied in \cite{Eisinberg}. 
The relationship between $ \kappa_2 (A) $ and $ \kappa_F (A) $ is expressed by the inequality
$$ \kappa_F (A) \geq \kappa_2 (A) + \frac{1}{\kappa_2 (A)} + m - 2 ,$$
see \cite{Bazan}, which implies that the ill-conditioning associated with the monomial basis and proved in Theorem \ref{th:cheby tool2} immediately carries over to the Frobenius setting.

\subsubsection{The Frobenius conditioning for Chebyshev segments}
Proposition \ref{prop:orthogonality} and subsequent computations allow to deduce results beyond the conditioning in the Euclidean norm for the histopolation matrix $ V_d $. 
First, observe that the Cauchy-Schwarz inequality applied to \eqref{eq:Frobcond} gives
\begin{equation} \label{eq:FrobeniusLB}
    \kappa_F (V_d) = \Vert V_d \Vert_F \Vert V_d^{-1} \Vert_F = \sqrt{\sum_{i=1}^d \sigma_i^2 (V_d)} \sqrt{\sum_{i=1}^d \frac{1}{\sigma_i^2 (V_d)}} \geq \sum_{i=1}^d 1 = d ,
\end{equation}
which yields a linear lower bound for $ \kappa_F ( V_d) $. For large $ d $, we are also in a position to provide an asymptotic upper bound for $ \kappa_F (V_d) $. To this end, we shall need the integral sine function
$$ \mathrm{Si} (z) := \int_0^z \frac{\sin (t)}{t} \de t .$$

\begin{lemma} \label{lem:Frobnorm}
Let $ 0 < a < 1 $. One has
	\begin{equation} \label{eq:boundVF}
		\lim_{d \to \infty} \Vert V_d \Vert_F = \sqrt{2a \pi \mathrm{Si} (2a \pi) - 2 \sin^2 (a \pi)} .
	\end{equation}
\end{lemma}

\begin{proof}
	First of all, recall that $ \Vert V_d \Vert_{F}^2 = \mathrm{Tr} \left( V_d^\top V_d \right) $. Hence, by Proposition \ref{prop:orthogonality}, substituting $ \rho = a \pi /d $, we get
	$$ \Vert V_d \Vert_{F}^2 = \mathrm{Tr} \left(V_d^\top V_d\right) = \sum_{i=1}^{d-1} \frac{2d}{i^2} \sin^2 (i \rho) + \frac{4}{d} \sin^2 (d \rho)  = \sum_{i=1}^{d} \frac{2d}{i^2} \sin^2 \left( i \frac{a}{d} \pi \right) + \frac{2}{d} \sin^2 (a \pi). $$
	The terms of this sum are positive and decreasing to $ 0 $, and the sum converges due to the presence of the term $ i^2 $. We may hence interpret it as a Riemann integral where $ \Delta x := 1/d $ and $ x_i := i/d $ and passing to the limit for $ d \to \infty $. Thus, one has
	\begin{align*} 
		& \lim_{d \to + \infty} \left( \sum_{i=1}^{d} \frac{2d}{i^2} \sin^2 \left( i \frac{a}{d} \pi \right) + \frac{2}{d} \sin^2 (a \pi) \right)
		=  \lim_{d \to + \infty} \sum_{i=1}^d \frac{2d}{i^2} \sin^2 \left( a \frac{i}{d} \pi \right) \\
        = & \lim_{d \to + \infty} \sum_{i=1}^d \frac{2d^2}{i^2 d} \sin^2 \left( a \frac{i}{d} \pi
        \right) 
		= \lim_{d \to + \infty} 2 \sum_{i=1}^d \frac{1}{x_i^2} \sin^2 \left( a \pi x_i \right) \Delta x = 2 \int_{0}^1 \frac{\sin^2 \left( a \pi x \right)}{x^2} \de x .
	\end{align*}
	This integral can be expressed using the sine integral function $ \mathrm{Si} (z) $, and after the immediate change of variable $ t = a \pi x $ and applying integration by parts, one obtains
	\begin{align*} 2 \int_{0}^1 \frac{\sin^2 \left( a \pi x\right)}{x^2} \de x & = 2 a \pi \int_0^{a\pi } \frac{\sin^2 (t)}{t^2} \de t = 2 a\pi \left( \mathrm{Si} (2a \pi) - \frac{\sin^{2} ( a \pi)}{a \pi} \right) \\
	& = 2a \pi \mathrm{Si} (2a \pi) - 2 \sin^2 (a \pi).
	\end{align*}
	Taking the square root we get the claim.
\end{proof}

The quantity \eqref{eq:boundVF} can be numerically estimated, thus Lemma \ref{lem:Frobnorm} can concretely be used to describe the asymptotics of the first term of the right hand side of \eqref{eq:Frobcond}. Further, due to the monotonicity as a function of $ a $ for $ 0 < a < 1 $, one concludes that 
$$ \lim_{d \to \infty} \Vert V_d \Vert_F \leq \sqrt{2 \pi \mathrm{Si} (\pi)} ,$$ obtaining an upper bound for $ \Vert V_d \Vert_F $ which is independent of $ a $.

To obtain an upper bound for $ \Vert V_d^{-1} \Vert_F $ and hence an estimate for the remaining term of \eqref{eq:Frobcond}, we exploit the explicit expression for $ V_d^{-1} $ derived in Section \ref{sect:inverseC2}.

\begin{lemma} \label{lem:Frobnorminverse}
    Let $ 0 < \rho < \pi/d $. One has
    $$ \left\Vert V_d^{-1} \right\Vert_F = \sqrt{\frac{1}{2d} \sum_{i=1}^{d-1} \left( \frac{i}{\sin (i \rho)} \right)^2 + \frac{d}{4 \sin^2 (d \rho)}} ,$$
    whence, for $ 0 < \rho  < \pi / d $,
    \begin{equation}
        \left\Vert V_d^{-1} \right\Vert_F \leq \frac{d}{\sqrt{2} \sin (d \rho)} = \frac{d}{\sqrt{2} \sin (a \pi)} .
    \end{equation}
\end{lemma}

\begin{proof}
    By applying the definition of Frobenius norm, and working with squared quantities,
    \begin{align*}
    \left\Vert V_d^{-1} \right\Vert_F^2 = \sum_{i=1}^d \sum_{j=1}^d \left[ V_{i,j}^{-1} \right]^2 = \sum_{i=1}^{d-1} \sum_{j=1}^d \left[ V_{i,j}^{-1} \right]^2 + \sum_{j=1}^d \left[ V_{d,j}^{-1} \right]^2
    \end{align*}
    whence, plugging the explicit terms given in Section \ref{sect:inverseC2}, one finds
    \begin{align*}
    \left\Vert V_d^{-1} \right\Vert_F^2 & = \sum_{i=1}^{d-1} \sum_{j=1}^d \left( \frac{i \sin (i \tau_j)}{d \sin (i \rho)} \right)^2 + \sum_{j=1}^d \left( \frac{\sin (d \tau_j)}{2 \sin (d \rho)} \right)^2 \\
    & = \sum_{i=1}^{d-1} \frac{i^2}{d^2 \sin^2(i \rho)} \sum_{j=1}^d \sin^2 (i \tau_j) + \frac{1}{4 \sin^2 (d \rho)} \sum_{j=1}^d \sin^2 (d \tau_j) .
    \end{align*}
    Both sums depending on $ j $ have been computed in \eqref{eq:orthsine}. In particular, the first equals $ d/2 $, while the second equals $ d $. Substistuting, this gives
    $$ \left\Vert V_d^{-1} \right\Vert_F^2 = \sum_{i=1}^{d-1} \frac{i^2}{d^2 \sin^2(i \rho)} \frac{d}{2} + \frac{1}{4 \sin^2 (d \rho)} d = \frac{1}{2d} \sum_{i=1}^{d-1} \frac{i^2}{\sin^2(i \rho)} + \frac{d}{4 \sin^2 (d \rho)} .$$
    Taking the square root one has the first equality. Exploiting the concavity of the sine function in $ (0, \pi) $, Lemma \ref{lem:decreasing} extends to $ j = d $ and we get that $ \sin (x) / x \geq \sin (d \rho) / (d \rho) $ in $ (0, d \rho) $, whence
    $$ \frac{i}{\sin (i \rho)} \leq \frac{d}{\sin (d \rho)} $$
    for all the terms in the remaining sum. We may thus expand the above equation to get
    \begin{align*}
        \left\Vert V_d^{-1} \right\Vert_F^2 & = \frac{1}{2d} \sum_{i=1}^{d-1} \frac{i^2}{\sin^2(i \rho)} + \frac{d}{4 \sin^2 (d \rho)} \leq \frac{1}{2d} \sum_{i=1}^{d-1} \frac{d^2}{\sin^2 (d \rho)} + \frac{d}{4 \sin^2 (d \rho)} \\
        & = \frac{d(d-1)}{2 \sin^2 (d \rho)} + \frac{d}{4 \sin^2 (d \rho)} = \frac{2d^2 - d}{4 \sin^2 (d \rho)} \leq \frac{d^2}{2 \sin^{2} ( d \rho) } .
    \end{align*}
    Taking the square root and substituting $ \rho = a \pi /d $, one also obtains $ \left\Vert V_d^{-1} \right\Vert_F \leq \frac{d}{\sqrt{2} \sin (a \pi)} $. 
\end{proof}

Merging Lemma \ref{lem:Frobnorm} and Lemma \ref{lem:Frobnorminverse} with the general lower bound \eqref{eq:FrobeniusLB}, we readily obtain the main result of the section.

\begin{theorem} \label{thm:condFrobenius}
	The conditioning $ \kappa_F (V_d) $ grows asymptotically linearly with $ d $. In particular, for $ 0 < a < 1 $, one has
    \begin{equation} \label{eq:linearFrobthm}
    d \leq \kappa_F (V_d) \leq \frac{\sqrt{a \pi \mathrm{Si}(2 a \pi) - \sin^2 (a \pi)}}{\sin(a \pi)} d .
    \end{equation}
\end{theorem}
%
%

Figure \ref{fig:condFro} depicts the trend of the Frobenius conditioning, predicted by multiplying the bounds given in Lemma \ref{lem:Frobnorm} and Lemma \ref{lem:Frobnorminverse}, which appears indeed linear. For small $ a $, by performing a Taylor expansion, one sees that the right hand side of \eqref{eq:linearFrobthm} is indeed close to $ d $. 

%
\begin{figure}[h!]
	\centering
    \includegraphics[width=0.7\textwidth]{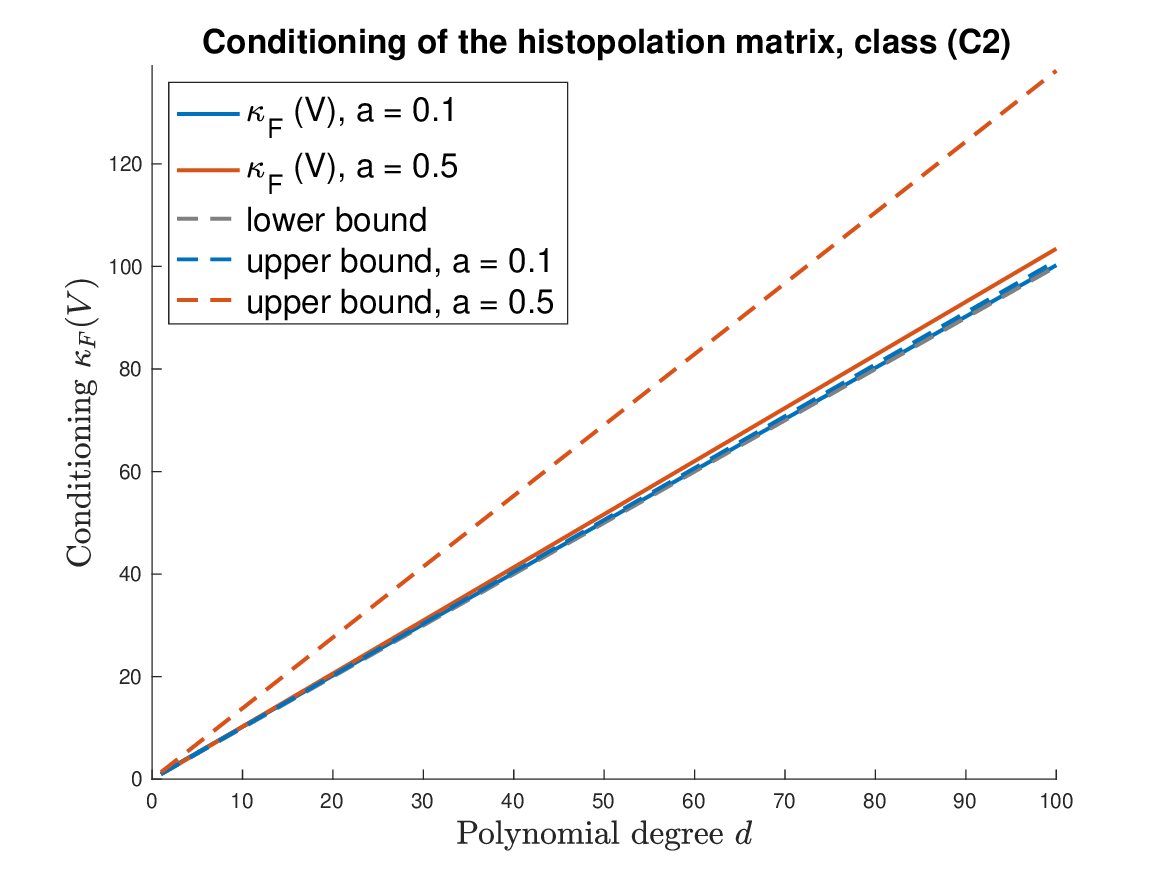}
	\caption{Linear conditioning in the Frobenius norm of the Vandermonde matrix in the Chebyshev basis for  $ a = 0.1 $ and $ a = 0.5 $, with the respective upper bounds.} \label{fig:condFro}
\end{figure}

\section*{Conclusions and further directions}

In the present paper, we have studied the conditioning of Vandermonde matrices $ V_d $ arising in histopolation. We focused on the Euclidean conditioning $  \kappa_2 (V_d) $, and we were able to identify two antipodal behaviours: the first one is exponential with $ d $, and is peculiar to the monomial basis. While this basis is somehow sensitive to the underlying collection of supports, the general trend is indeed exponential, and hence gives a neat example of unifying exponential ill-conditioning, in all the $\|\cdot\|_{l^k\rightarrow l^p}$ norms, $k,p\in [1,\infty]$ (the case $k=p=2$ being the standard Euclidean or spectral norm). On the other hand, we recognized that the use of the Chebyshev basis can mitigate this effect; in particular, in combination with a particular class of segments, the conditioning drops and the Vandermonde matrix becomes very well-conditioned. As a by-product, owing to a discrete orthogonality, we also proved that this latter situation is favourable also with respect to the conditioning in the Frobenius norm $ \kappa_F(V_d) $.

Although the conditioning with respect to the Euclidean norm is the most common in numerical analysis, it is important to point out that it is not the only one. For instance, the conditioning with respect to the $ \infty $-norm is usually taken into account when dealing with Bernstein-Vandermonde collocation matrices \cite{Delgado}, whose numerical features are not only of theoretical interest, but also have several applications \cite{Allen}. Thus, a natural extension to this work consists in considering the Bernstein basis (fixing the reference interval to $ I = [-1,1] $) and analyzing the relative conditioning in the $ \infty $-norm. A well-established and interesting starting point for a  worthwhile comparison with the literature is represented by the work \cite{Farouki}.


\end{document}